\documentclass[onefignum,onetabnum]{siamart190516}

\usepackage{lipsum}
\usepackage{amssymb}
\usepackage{amsfonts}
\usepackage{graphicx}
\usepackage{epstopdf}
\usepackage{algorithmic}
\usepackage{enumitem}

\usepackage{tikz}
\usepackage{pgfplots}
\usetikzlibrary{shapes.arrows, patterns, calc}
\usepackage{tikz-3dplot}
\usepackage{bbm}
\ifpdf
  \DeclareGraphicsExtensions{.eps,.pdf,.png,.jpg}
\else
  \DeclareGraphicsExtensions{.eps}
\fi

\definecolor{lightblue}{HTML}{A1BDC7}
\definecolor{orange}{HTML}{D98C21}
\definecolor{silver}{HTML}{B0ABA8}
\definecolor{rust}{HTML}{B8420F}
\definecolor{seagreen}{HTML}{2E6B69}
\definecolor{joshua}{HTML}{FBDC7F}
\definecolor{darksky}{HTML}{154c79}

\colorlet{lightsilver}{silver!30!white}
\colorlet{darkorange}{orange!85!black}
\colorlet{darksilver}{silver!85!black}
\colorlet{darklightblue}{lightblue!85!black}
\colorlet{darkrust}{rust!85!black}
\colorlet{darkseagreen}{seagreen!85!black}

\hypersetup{colorlinks=true,linkcolor=darkrust,citecolor=darkseagreen,urlcolor=darksilver}

\usepackage{mathtools}
\newcommand*{\Reals}{\mathbb{R}}
\newcommand*{\Naturals}{\mathbb{N}}
\newcommand*{\dx}{\,\mathrm{d}x}

\renewcommand*{\O}{\mathcal{O}}

\newcommand*{\dist}{\operatorname{dist}}
\newcommand*{\trunc}{\operatorname{trunc}}

\newcommand*{\argmin}{\operatorname{argmin}}

\newcommand*{\colors}{\mathcal{C}}
\newcommand*{\supercolors}{\tilde{\mathcal{C}}}

\makeatletter
\newcommand{\subalign}[1]{%
  \vcenter{%
    \Let@ \restore@math@cr \default@tag
    \baselineskip\fontdimen10 \scriptfont\tw@
    \advance\baselineskip\fontdimen12 \scriptfont\tw@
    \lineskip\thr@@\fontdimen8 \scriptfont\thr@@
    \lineskiplimit\lineskip
    \ialign{\hfil$\m@th\scriptstyle##$&$\m@th\scriptstyle{}##$\hfil\crcr
      #1\crcr
    }%
  }%
}
\makeatother

\newcommand*{\K}{G}
\newcommand*{\KM}{\mathbf{\Theta}}
\newcommand*{\IK}{\mathcal{L}}
\newcommand*{\IKM}{\mathbf{A}}
\newcommand*{\rhsM}{\mathbf{b}}
\newcommand*{\solM}{\mathbf{u}}
\newcommand*{\dv}{\mathbf{v}}
\newcommand*{\de}{\mathbf{e}}
\newcommand*{\dw}{\mathbf{w}}
\newcommand*{\cw}{w}
\newcommand*{\cW}{W}
\newcommand*{\cv}{v}
\newcommand*{\cV}{V}
\newcommand*{\dV}{\mathbf{V}}
\newcommand*{\dtau}{\boldsymbol{\tau}}
\newcommand*{\dt}{\boldsymbol{t}}
\newcommand*{\dW}{\mathbf{W}}
\newcommand*{\dM}{\mathbf{M}}
\newcommand*{\dm}{\mathbf{m}}
\newcommand*{\disco}{\mathbf{o}}

\newcommand*{\dL}{\mathbf{L}}
\newcommand*{\dO}{\mathbf{O}}
\newcommand*{\oracle}{\omega}

\renewcommand*{\S}{\mathcal{S}}
\newcommand*{\schur}{\mathbb{S}}
\newcommand*{\addschur}{\mathbb{A}}
\newcommand*{\I}{I}

\newcommand*{\J}{J}
\newcommand*{\rhs}{b}

\newcommand*{\fro}{\operatorname{Fro}}
\newcommand{\poly}{\operatorname{poly}}

\newsiamremark{remark}{Remark}
\newsiamremark{hypothesis}{Hypothesis}
\newsiamremark{example}{Example}
\crefname{hypothesis}{Hypothesis}{Hypotheses}
\newsiamthm{claim}{Claim}

\headers{Cholesky recovery of elliptic solvers}{F. Sch{\"a}fer and H. Owhadi}

\title{Sparse recovery of elliptic solvers from matrix-vector products}

\author{%
  Florian\ Sch{\"a}fer\thanks{Georgia Institute of Technology, S1317 CODA, 756 W Peachtree St Atlanta, GA 30332, \newline \email{fts@gatech.edu}, Corresponding Author} 
  \and
  Houman Owhadi\thanks{Department of Computing + Mathematical Sciences, California Institute of Technology}
}

\usepackage{amsopn}
\DeclareMathOperator{\diag}{diag}

\ifpdf 
\hypersetup{
  pdftitle={Sparse recovery of elliptic solvers from matrix-vector products},
  pdfauthor={F. Sch{\"a}fer and H. Owhadi}
}
\fi

\begin{document}

\maketitle

\begin{abstract}
  In this work, we show that solvers of elliptic boundary value problems in $d$ dimensions can be approximated to accuracy $\epsilon$ from only $\O\left(\log(N)\log^{d}(N / \epsilon)\right)$ matrix-vector products with carefully chosen vectors (right-hand sides).  The solver is only accessed as a black box, and the underlying operator may be unknown and of an arbitrarily high order.
  Our algorithm (1)  has complexity $\O\left(N\log^2(N)\log^{2d}(N / \epsilon)\right)$ and represents the solution operator as a sparse Cholesky factorization with  $\O\left(N\log(N)\log^{d}(N / \epsilon)\right)$ nonzero entries, (2) allows for embarrassingly parallel evaluation of the solution operator and the computation of its log-determinant, (3) 
  allows for $\O\left(\log(N)\log^{d}(N / \epsilon)\right)$ complexity computation of individual entries of the matrix representation of the solver that, in turn, enables its recompression to an $\O\left(N\log^{d}(N / \epsilon)\right)$ complexity representation.    As a byproduct, our compression scheme produces a homogenized solution operator with near-optimal approximation accuracy. 
  By polynomial approximation, we can also approximate the continuous Green's function (in operator and Hilbert-Schmidt norm) to accuracy $\epsilon$ from $\mathcal{O}\left(\log^{1 + d}\left(\epsilon^{-1}\right)\right)$ solutions of the PDE. We include rigorous proofs of these results. 
  To the best of our knowledge, our algorithm achieves the best known trade-off between accuracy $\epsilon$ and the number of required matrix-vector products.
\end{abstract}



\section{Introduction}
\subsection{Fast solvers are not enough}
Linear elliptic partial differential equations (PDEs) $\IK u = \rhs$ are ubiquitous in engineering, physics, and machine learning.
After discretization, the solution of these equations amounts to solving a linear system of $N$ equations $\IKM \solM = \rhsM$. 
In this work, we assume that we have access to the PDE through a black box $\oracle(\rhs) \coloneqq \IK^{-1} \rhs$ or $\oracle(\rhsM) \coloneqq \IKM^{-1} \rhsM$ that solves the PDE for arbitrary right-hand sides $\rhs$ or $\rhsM$.
In practice, $\oracle$ could be implemented by a legacy solver or physical experiments. 
While $N$ invocations of $\oracle$ are enough to fully reconstruct $\IKM^{-1}$ and therefore $\IKM$, this is prohibitively expensive since invoking $\oracle$ for a single right-hand side has a computational cost of at least $N$ (to write down the result) and may require a physical experiment or extensive computation (like the solution of a PDE).

\subsubsection{Our contribution} For $\Omega \subset \Reals^d$ a Lipschitz bounded domain, $0 < s  \in \Naturals$, and $\IK:H_0^s(\Omega) \longrightarrow H^{-s}(\Omega)$, let $\IK$ be linear, invertible, continuous, local, positive, and self-adjoint. In this setting, we show how to obtain an $\epsilon$-accurate sparse Cholesky factorization of $\IKM^{-1}$ using only $\O\big(\log(N) \log^d(N / \epsilon)\big)$ invocations of $\oracle$. 
By choosing $N \approx \poly(\epsilon^{-1})$, we recover the solution operator of $\IK$ (in operator and Hilbert-Schmidt norm) to accuracy $\epsilon$ from solutions for $\mathcal{O}\big(\log^{1 + d}\left(\epsilon^{-1}\right)\big)$ right-hand sides.
Our algorithm has a straightforward extension that recovers the LU factorization of a non-self-adjoint operator from matrix-vector and matrix-transpose-vector products. 
We conjecture that theoretical results extend to LU factorizations of elliptic operators where only the leading order term is self-adjoint and positive.
Our results have implications for the following theoretical and practical questions.

\subsubsection{Operator learning} A growing body of work is concerned with learning solution operators of partial differential equations from data consisting of input-output pairs $\left(\rhsM, \omega(\rhsM)\right)$.
The proposed work achieves an exponential improvement in the sample complexity of provably accurate learning of linear elliptic solution operators. Furthermore, our complexity vs accuracy error analysis is performed in the strict worst-case setting rather than the more lenient average-case setting (which benefits from concentration of measure effects).

\subsubsection{Parallelization} A legacy solver $\oracle$ may not easily parallelize across large numbers of workers.
The matrix-vector product application of the sparse Cholesky factorization obtained by our approach is embarrassingly parallel.

\subsubsection{Reduced order modeling} \label{sec:reduced_order}Computing a rank-$k$ approximation of $\IKM^{-1}$ using the power method requires more than $k$ applications of $\oracle$. 
In contrast, we are able to compute rank-$k$ approximations at a cost of $\O\left(\log^{d + 1}(k)\right)$ applications of $\oracle$.

\subsubsection{Applying submatrices of \texorpdfstring{$\IK^{-1}$}{.}} \label{sec:submat} Many of the right-hand sides to which we apply the legacy solver may be sparse, or we may only be interested in some components of the solution $\solM$. 
For instance, we may be interested in efficiently individual entries of $\IKM^{-1}$ to compute the resistance metric associated with $\IKM$.
Computing only the diagonal entries of $\IKM^{-1}$ requires $N$ applications of $\oracle$, likely resulting in a computational cost of at least $N^2$. 
In contrast, our representation of $\IKM^{-1}$ allows computing individual entries of $\IKM^{-1}$ at complexity $\O\big(\log(N) \log^d(N / \epsilon)\big)$.
Thus, it can be compressed to a sparse factorization of $\IKM$ with $\O\big(N \log^d(N / \epsilon)\big)$ nonzero entries \cite{schafer2020sparse}.

\subsubsection{Gaussian process priors} Elliptic PDEs can be employed to define Gaussian \emph{smoothness priors}\footnote{We emphasize that our method is fully capable of dealing with PDEs with rough coefficients that have Green's functions of only finite orders of smoothness.} with covariance matrix $\IKM^{-1}$. 
To draw samples from this Gaussian process, we need to apply \emph{a square root} of $\IKM^{-1}$ to an i.i.d. Gaussian vector, and for tuning hyperparameters, we need the log-determinant of $\IKM^{-1}$.
The Cholesky factorization of $\IKM^{-1}$ that we compute allows performing these operations.

\subsection{Related work} Multiple approaches address some of the above desiderata.
The following review covers the most closely related examples from the vast literature on numerical methods for elliptic partial differential equations.
The proposed method learns solution operators from data instead of solving a PDE. 
We thus do not review fast solvers for elliptic problems, referring to \cite[Section 1.2]{schafer2021compression}.

\subsubsection*{Selected inversion and homogenization} The work of \cite{lin2009fast,lin2011selinv} on selected inversion aims to perform the tasks described in \cref{sec:submat} for inverses of arbitrary sparse matrices. 
Similarly, \emph{numerical homogenization} and related fields (see \cite{altmann2021numerical} for a survey) are concerned with computing reduced models for inverses of partial differential operators, thus addressing the task described in section \cref{sec:reduced_order}.
In contrast to our proposed approach, these methods require explicit knowledge of the partial differential operator or access to its matrix representation.
\subsubsection*{Learning operators with Neural Networks and operator-valued kernels}
The problem of \emph{learning} the solution operators of a partial differential equation from data (in the form of pairs of right-hand sides and solutions) provided by an existing solver has received considerable attention. 
Some works, such as \cite{fan2019bcr,fan2019multiscale,fan2019multiscaleH,li2020neural,li2020fourier,li2020multipole} use deep neural networks with architectures inspired by conventional fast solvers. Others are motivated by universal approximation theorems for operators \cite{lu2019deeponet}, or
employ random feature approximations of operator-valued kernels \cite{nelsen2021random}.
These methods can also be seen as regression methods with (possibly data-dependent) operator-valued kernels \cite{owhadi2020ideas}.
They can be applied to \emph{arbitrary} nonlinear operators, but their convergence (or lack thereof) is poorly understood in the nonlinear setting.
\subsubsection*{Learning linear operators} 
More closely related to the present work are methods providing theoretical guarantees for learning structured linear operators.
\cite{de2021convergence} approximates operators between infinite-dimensional Hilbert spaces from noisy measurements using a Bayesian approach and characterizes convergence (with randomized right-hand-sides), in terms of the spectral decay of the target operator (assuming the operator to be self-adjoint and diagonal in a basis shared with the Gaussian prior and noise covariance).
\cite{stepaniants2021learning} seeks to learn the Green's function of an elliptic PDE by empirical risk minimization over a reproducing kernel Hilbert space, characterizing convergence in terms of spectral decay.
The randomized algorithm of \cite{boulle2021learning} comes with rigorous bounds on the number of input-output measurements required to learn solution operators of linear elliptic PDEs. 
The convergence rates of these methods are limited by the spectral decay of the operator, and the number of required matrix-vector products scales polynomially with the target accuracy.
Alternatively, \cite{lin2011fast,martinsson2016compressing,ambartsumyan2020hierarchical,levitt2022randomized} target operators with hierarchical off-diagonal low-rank structure, allowing them to overcome the limitations of global low-rank structure. 
Finally, \cite{feischl2018sparse} compresses expectations of solution operators of stochastic PDEs from samples of the entire operator, and \cite{martinsson2008rapid} combines pairwise evaluations with matrix-vector products to construct hierarchical matrix approximations of elliptic solution operators.

\subsection{Our method in a nutshell}
Our method uses three ingredients to get away with a logarithmic (in both $N$ \emph{and} $\epsilon$) number of matrix-vector products.
\textbf{1:} \cite{schafer2021compression} proves that Cholesky factors of solution operators of elliptic PDEs are sparse with known sparsity pattern, up to exponentially small entries (\cref{fig:sparse_cholesky_factor}). \textbf{2:} Sparse matrices and Cholesky factorizations with leading columns having disjoint nonzero sets can be obtained from a single matrix-vector product with a cleverly chosen vector (top of \cref{fig:graph_coloring}).
\textbf{3:} After identifying the leading columns, we can subtract them from a Cholesky factorization to recover the later columns efficiently (bottom of \cref{fig:graph_coloring}). 

The Cholesky factors are only \emph{approximately} sparse. 
The main theoretical contribution of this work is a rigorous bound on the error propagation during the procedure outlined above, leading to the accuracy-vs-complexity results advertised in the abstract. 
This exponential accuracy is vastly better than what can be hoped for when using global low-rank approximations such as \cite{de2021convergence,boulle2021learning,stepaniants2021learning}. 
The principles underlying \cite{lin2011fast} are closely related to \textbf{2}, \textbf{3} above.
However, \cite{lin2011fast} does not provide rigorous bounds on the error propagation, and thus, their theoretical results do not cover elliptic solution operator \emph{in the wild}. 
Even ignoring this aspect, the improvements of sparse Cholesky factors over hierarchical matrices in approximating elliptic solution operators \cite[Sections 1.3 and 7.1]{schafer2021compression} result in asymptotically lower computational cost and fewer required matrix-vector products.
Experiments in \cref{sec:numerics} show that our method requires fewer matrix-vector products than reported by \cite{lin2011fast}.

\begin{figure}
    \centering
    \includegraphics[scale=0.675]{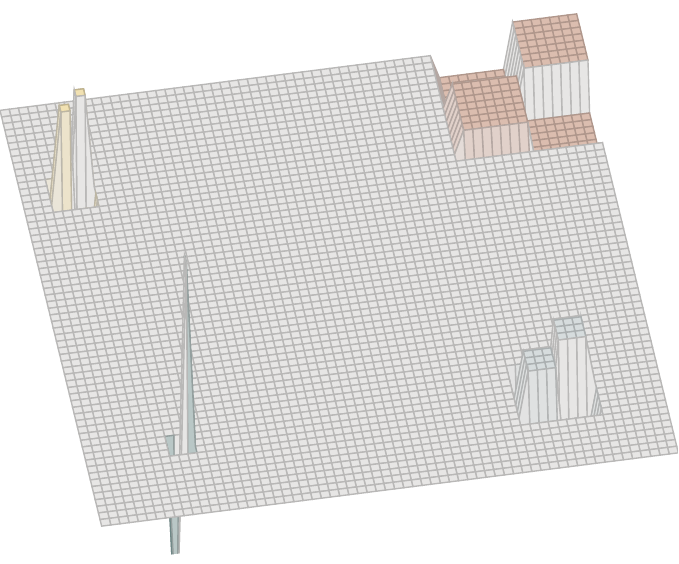}
    \includegraphics[scale=0.15]{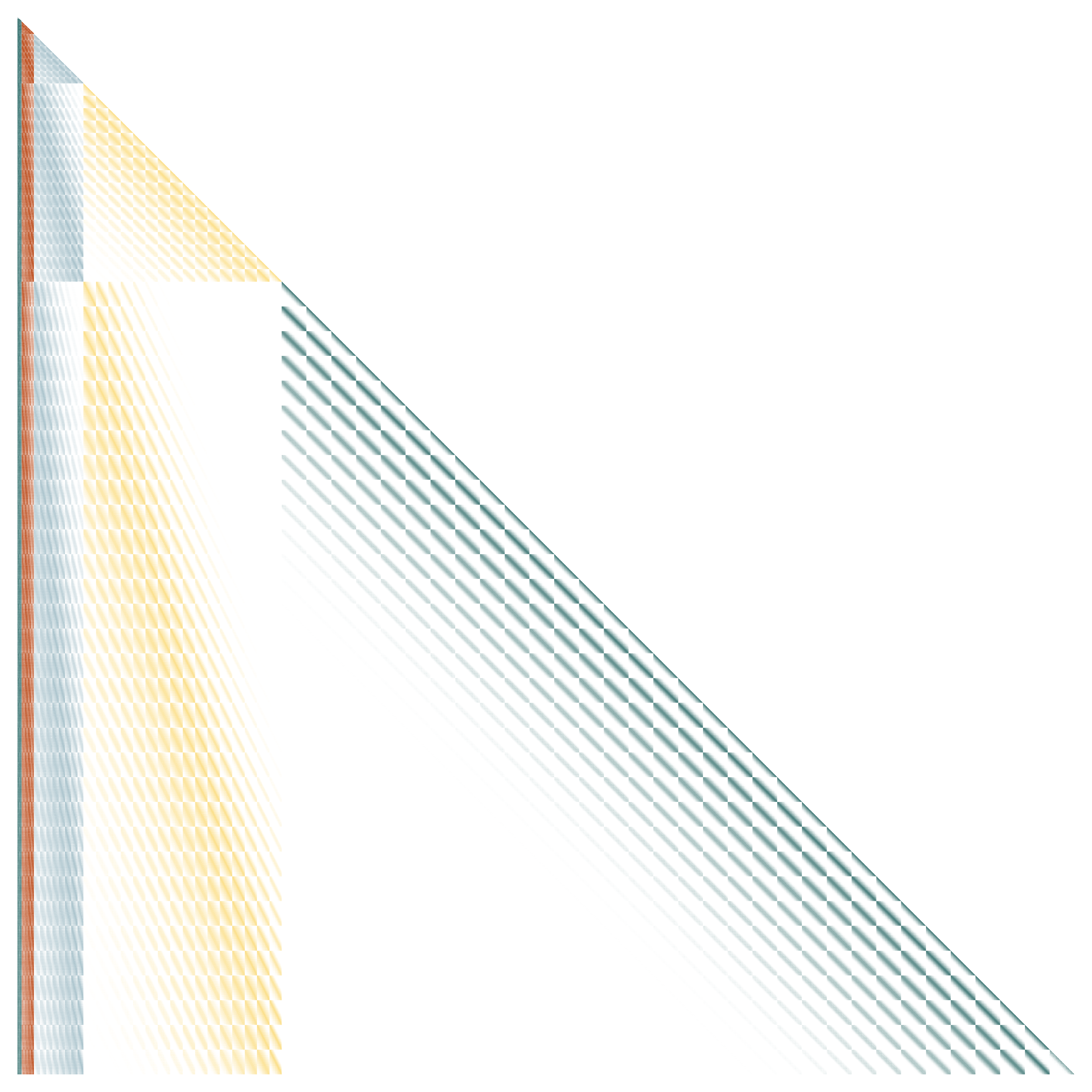}
    \hspace{-25pt}
    \includegraphics[scale=0.675]{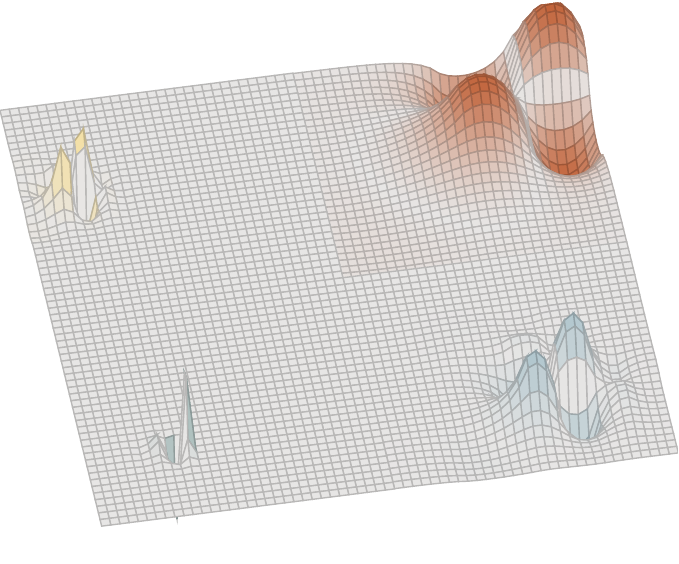}
  \caption{Left: Four examples basis functions of a Haar-type multiresolution basis, on four different levels. Center: Heatmap of Cholesky factors of $\KM$ expressed in the multiresolution basis with diagonal scaled to one. Magnitudes of entries shown in white are smaller than $10^{-10}$.
  The colors of the entries in each block column match that of the example basis function on that level. Right: Columns of Cholesky factors for basis functions on the left, interpreted as spatial functions.
  The near-sparsity of the Cholesky factors is shown in \cite{schafer2021compression}. Intuitively, it arises since Gaussian elimination of the earlier degrees (which are coarse-scale wavelets) removes the long-range interactions.}
  \label{fig:sparse_cholesky_factor}
\end{figure}

\begin{figure}
  \centering
  \includegraphics[scale=0.650]{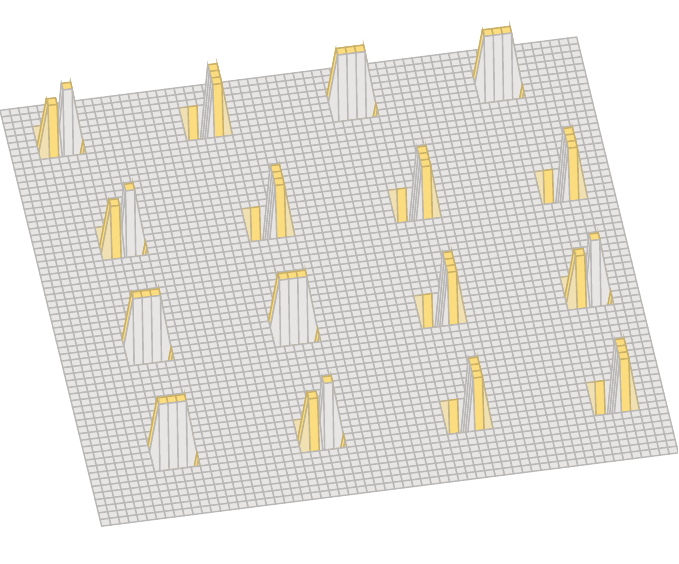}
  \includegraphics[scale=0.13]{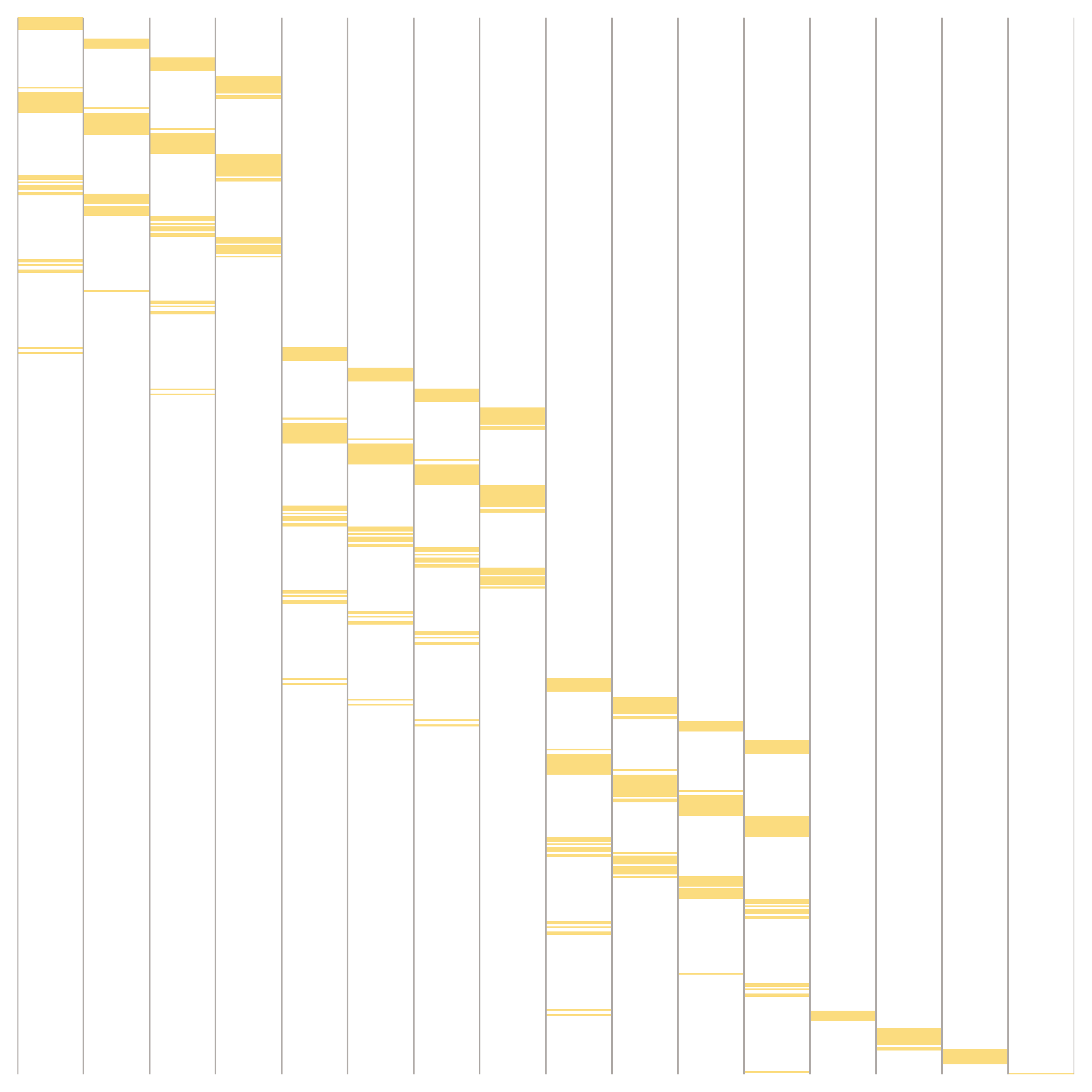}
  \includegraphics[scale=0.650]{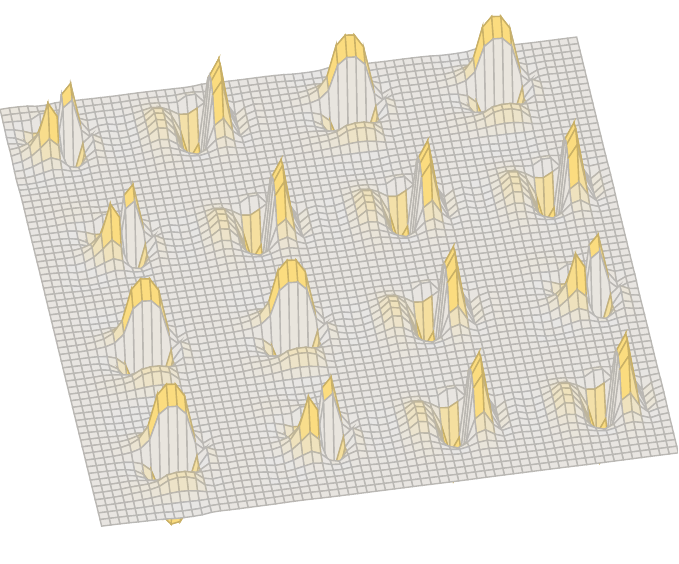}
  \caption{We illustrate an example right-hand side (left) and the resulting matrix-vector product (right) after subtracting coarse-scale contributions. 
  In the center, we illustrate part of the sparsity pattern (truncating entries smaller than $10^{-3}$) of the corresponding columns of the Cholesky factor.}
  \label{fig:multicolor_cholesky}
\end{figure}

\begin{figure}
  \centering
  \begin{minipage}{0.39 \textwidth}
    \centering
    \begin{tikzpicture}[scale=0.45]
\begin{scope}
	\begin{scope}[xscale=1.3, yscale=1.3]
		\draw (-3, 5) rectangle (0, 0);
		\draw (-3.0, 5.0) rectangle (-2.5, 0.00);
		\draw[fill=silver] (-3.0, 4.0) rectangle (-2.5, 3.5);
		\draw[fill=silver] (-3.0, 2.0) rectangle (-2.5, 0.0);

		\draw (-2.5, 5.0) rectangle (-2.0, 0.00);
		\draw[fill=orange] (-2.5, 5.0) rectangle (-2.0, 4.5);
		\draw[fill=orange] (-2.5, 2.5) rectangle (-2.0, 1.5);
		\draw[fill=orange] (-2.5, 1.0) rectangle (-2.0, 0.5);

		\draw (-2.0, 5.0) rectangle (-1.5, 0.00);
		\draw[fill=silver] (-2.0, 4.5) rectangle (-1.5, 3.0);
		\draw[fill=silver] (-2.0, 2.5) rectangle (-1.5, 1.5);
		\draw[fill=silver] (-2.0, 1.0) rectangle (-1.5, 0.5);

		\draw (-1.5, 5.0) rectangle (-1.0, 0.00);
		\draw[fill=rust] (-1.5, 4.5) rectangle (-1.0, 4.0);
		\draw[fill=rust] (-1.5, 3.0) rectangle (-1.0, 2.5);
		\draw[fill=rust] (-1.5, 0.5) rectangle (-1.0, 0.0);

		\draw (-1.0, 5.0) rectangle (-0.5, 0.00);
		\draw[fill=lightblue] (-1.0, 4.0) rectangle (-0.5, 3.0);
		\draw[fill=lightblue] (-1.0, 1.5) rectangle (-0.5, 1.0);

		\draw (-0.5, 5.0) rectangle (-0.0, 0.00);
		\draw[fill=silver] (-0.5, 5.0) rectangle (-0.0, 4.0);
		\draw[fill=silver] (-0.5, 3.5) rectangle (-0.0, 2.5);
		\draw[fill=silver] (-0.5, 2.0) rectangle (-0.0, 1.0);
	\end{scope}
\end{scope}

\begin{scope}[xscale=1.3, yscale=1.3]
\begin{scope}[xshift=-0.10cm]
\node at (1.0, 2.5) {\Huge \textcolor{darksilver}{$\cdot$}};
\end{scope}
\end{scope}

\begin{scope}[xscale=1.3, yscale=1.3]
\begin{scope}[xshift=-0.25cm, yshift=-0.0cm]
	\draw (2, 4) rectangle (2.5, 1);
	\node at (2.25, 3.75) {\small {$0$}};
	\draw (2, 4.0) rectangle (2.5, 3.5);
	\node at (2.25, 3.25) {\small {$1$}};
	\draw (2, 3.5) rectangle (2.5, 3.0);
	\node at (2.25, 2.75) {\small {$0$}};
	\draw (2, 2.5) rectangle (2.5, 2.0);
	\node at (2.25, 2.25) {\small {$1$}};
	\draw (2, 2.0) rectangle (2.5, 1.5);
	\node at (2.25, 1.75) {\small {$1$}};
	\draw (2, 1.5) rectangle (2.5, 1.0);
	\node at (2.25, 1.25) {\small {$0$}};
\end{scope}
\end{scope}

\begin{scope}[xscale=1.3, yscale=1.3]
\begin{scope}[xshift=-0.5cm,yshift=1.0]
\node at (4.0, 2.5) {\Huge \textcolor{darksilver}{$=$}};
\end{scope}
\end{scope}

\begin{scope}[xscale=1.3, yscale=1.3]
	\begin{scope}[xshift=-0.5cm]
		\draw (5.5, 5) rectangle (6.0, 0);

		\draw[fill=orange] (5.5, 5.0) rectangle (6.0, 4.5);
		\draw[fill=orange] (5.5, 2.5) rectangle (6.0, 1.5);
		\draw[fill=orange] (5.5, 1.0) rectangle (6.0, 0.5);

		\draw[fill=rust] (5.5, 4.5) rectangle (6.0, 4.0);
		\draw[fill=rust] (5.5, 3.0) rectangle (6.0, 2.5);
		\draw[fill=rust] (5.5, 0.5) rectangle (6.0, 0.0);

		\draw[fill=lightblue] (5.5, 4.0) rectangle (6.0, 3.0);
		\draw[fill=lightblue] (5.5, 1.5) rectangle (6.0, 1.0);
	\end{scope}
\end{scope}
	
    \end{tikzpicture} 
  \end{minipage}
  \hfill {\color{darksilver}\vline width 1.0mm} \hfill
  \begin{minipage}{0.49 \textwidth}
    \centering
    \begin{tikzpicture}[scale=0.45]
		  	\begin{scope}[xscale=1.3, yscale=1.3]
	\begin{scope}[yshift=-2.0cm, xshift=-1.0cm]
		\draw (-1.5, 3.5) rectangle (-1.0, 0.50);
		\draw[fill=orange, fill opacity=0.25] (-1.5, 3.5) rectangle (-1.0, 3.0);
		\node at (-1.25, 3.25) {\small {$1$}};
		\draw[fill=orange] (-1.5, 2.5) rectangle (-1.0, 1.5);

		\draw (-2.5, 4.5) rectangle (-2.0, 0.50);
		\draw[fill=rust, fill opacity=0.25] (-2.5, 4.5) rectangle (-2.0, 4.0);
		\node at (-2.25, 4.25) {\small {$1$}};
		\draw[fill=rust] (-2.5, 3.0) rectangle (-2.0, 2.5);
		\draw[fill=rust] (-2.5, 1.0) rectangle (-2.0, 0.5);

		\draw (-2.0, 4.0) rectangle (-1.5, 0.50);
		\draw[fill=lightblue, fill opacity=0.25] (-2.0, 4.0) rectangle (-1.5, 3.5);
		\node at (-1.75, 3.75) {\small {$1$}};
		\draw[fill=lightblue] (-2.0, 1.5) rectangle (-1.5, 1.0);

		\draw[fill=silver] (-0.5, 3.0) rectangle (-1.0, 0.5);
	\end{scope}
\end{scope}

\begin{scope}[xscale=1.3, yscale=1.3]
	\begin{scope}[xshift=2.0cm, yshift=1.0cm, yscale=-1, rotate=90]

		\draw (-1.5, 3.5) rectangle (-1.0, 0.50);
		\draw[fill=orange, fill opacity=0.25] (-1.5, 3.5) rectangle (-1.0, 3.0);
		\node at (-1.25, 3.25) {\small {$1$}};
		\draw[fill=orange] (-1.5, 2.5) rectangle (-1.0, 1.5);

		\draw (-2.5, 4.5) rectangle (-2.0, 0.50);
		\draw[fill=rust, fill opacity=0.25] (-2.5, 4.5) rectangle (-2.0, 4.0);
		\node at (-2.25, 4.25) {\small {$1$}};
		\draw[fill=rust] (-2.5, 3.0) rectangle (-2.0, 2.5);
		\draw[fill=rust] (-2.5, 1.0) rectangle (-2.0, 0.5);

		\draw (-2.0, 4.0) rectangle (-1.5, 0.50);
		\draw[fill=lightblue, fill opacity=0.25] (-2.0, 4.0) rectangle (-1.5, 3.5);
		\node at (-1.75, 3.75) {\small {$1$}};
		\draw[fill=lightblue] (-2.0, 1.5) rectangle (-1.5, 1.0);

		\draw[fill=silver] (-0.5, 3.0) rectangle (-1.0, 0.5);
	\end{scope}
\end{scope}

\begin{scope}[xscale=1.3, yscale=1.3]
\begin{scope}[xshift=1.0cm, yshift=-1.50cm]
\node at (1.0, 2.5) {\Huge \textcolor{darksilver}{$\cdot$}};
\end{scope}
\end{scope}

\begin{scope}[xscale=1.3, yscale=1.3]
\begin{scope}[xshift=1.0cm, yshift=-0.5cm]
	\draw (2, 3.5) rectangle (2.5, 0.5);
	\node at (2.25, 3.25) {\small {$1$}};
	\draw (2, 3.5) rectangle (2.5, 3.0);
	\node at (2.25, 2.75) {\small {$1$}};
	\draw (2, 3.0) rectangle (2.5, 2.5);
	\node at (2.25, 2.25) {\small {$1$}};
	\draw (2, 2.5) rectangle (2.5, 2.0);

	\draw (2, 2.0) rectangle (2.5, 1.5);
	\node at (2.25, 1.75) {\small {$0$}};

	\draw (2, 1.5) rectangle (2.5, 1.0);
	\node at (2.25, 1.25) {\small {$0$}};

	\draw (2, 1.0) rectangle (2.5, 0.5);
	\node at (2.25, 0.75) {\small {$0$}};

	\draw (2, 0.5) rectangle (2.5, 0.0);
	\node at (2.25, 0.25) {\small {$0$}};

	\draw (2, 0.0) rectangle (2.5, -0.5);
	\node at (2.25, -0.25) {\small {$0$}};
\end{scope}
\end{scope}

\begin{scope}[xscale=1.3, yscale=1.3]
	\begin{scope}[xshift=1.0cm,yshift=-1.5cm]
		\node at (4.0, 2.5) {\Huge \textcolor{darksilver}{$=$}};
	\end{scope}
\end{scope}

\begin{scope}[xscale=1.3, yscale=1.3]
	\begin{scope}[xshift=1.0cm, yshift=-1.50cm]
		\draw (5.5, 4.5) rectangle (6.0, 1.0);

		\draw[fill=orange, fill opacity=0.25] (5.5, 3.5) rectangle (6.0, 3.0);
		\node at (5.75, 3.25) {\small {$1$}};
		\draw[fill=orange] (5.5, 2.5) rectangle (6.0, 1.5);

		\draw[fill=rust, fill opacity=0.25] (5.5, 4.5) rectangle (6.0, 4.0);
		\node at (5.75, 4.25) {\small {$1$}};
		\draw[fill=rust] (5.5, 3.0) rectangle (6.0, 2.5);
		\draw[fill=rust] (5.5, 1.0) rectangle (6.0, 0.5);

		\draw[fill=lightblue, fill opacity=0.25] (5.5, 4.0) rectangle (6.0, 3.5);
		\node at (5.75, 3.75) {\small {$1$}};
		\draw[fill=lightblue] (5.5, 1.5) rectangle (6.0, 1.0);
	\end{scope}
\end{scope}

%
%
%
%
%
    \end{tikzpicture} 
  \end{minipage}
  \vfill
  \vspace{0.50cm}
  \vfill
  \begin{minipage}{1.00 \textwidth}
    \centering
    \begin{tikzpicture}[scale=0.45]
      \input{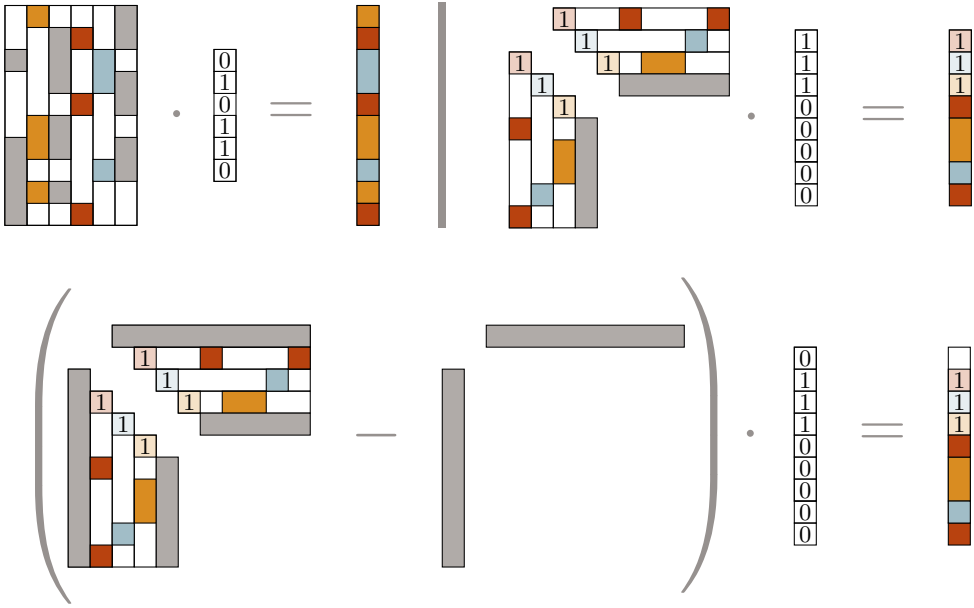}
    \end{tikzpicture}
  \end{minipage}
  \caption{\textbf{Top Left:} Columns (shown in orange, red, and blue) with disjoint and known sparsity patterns can be recovered by a single matrix-vector product with a carefully chosen vector. \\
  \textbf{Top Right:} Cholesky factorizations with disjoint leading-column sparsity patterns can be recovered in the same way. Unit diagonal entries simplify the presentation but are \underline{not necessary} for recovery.\\
  \textbf{Bottom:} If denser columns of the Cholesky factorization prevent recovery of sparser ones, we identify dense columns first and subtract their contribution to efficiently recover the sparser columns. \label{fig:graph_coloring}}
\end{figure}

\section{Our method}

\subsection{Summary} We now explain how to apply our method to inverses of an elliptic partial differential operator $\IK$ or its discretization $\IKM$, accessible through input-output maps $\rhs \mapsto \IK^{-1}\rhs$ or $\rhsM \mapsto \IKM^{-1}\rhsM$. 
Our method consists of four steps.

\begin{enumerate}[label=\textbf{(\Roman*)}]
  \item We construct a Haar-type multiresolution basis $\{w_{i}\}_{i \in \I}$ (resp. $\{\dw_{i}\}_{i \in \I}$)  indexed by a set  $\I$ of size $|\I| =N$, such that the Cholesky factor of $\KM \in \Reals^{\I \times \I}$ defined as $\KM_{ij} \coloneqq \int w_{i} \IK^{-1} w_{j}$ (resp. $\dw_{i}^{\top} \IKM^{-1} \dw_{j}$) is approximately sparse. \label{item:summary_multires}
  \item We choose a tuning parameter $\rho > 0$ and color the multiresolution basis functions such that two bases share the same color only if they are of the same scale $h^k$ and their supports are separated by a distance of at least $2 \rho h^{k-1}$. 
    We denote as $\colors$ the resulting partition of the multiresolution basis in different colors $c \in \colors$. 
    Therefore, each $c \in \colors$ is a subset of $\I$ consisting of those elements of $\I$ that are assigned to this color.
    The larger we choose $\rho$, the more accurate (and expensive) our approximation is. \label{item:summary_coloring}
  \item We construct a measurement matrix\footnote{Here and in the following, $\Reals^{\I}$ and $\Reals^{\I \times \J}$ denote vectors and matrices indexed by sets $\I$ and $\J$.} $\dM \in \Reals^{\I \times \colors}$ with columns $\dM_{:,c}$ given by sums of standard basis vectors  $\left\{\de_i\right\}_{i \in c}$  of a given color $c$.\footnote{The standard basis vector $\de_i$ has a $1$ as its $i$-th entry, with all other entries being $0$.} 
    We then construct the observation matrix $\dO \in \Reals^{\I \times \colors}$ containing the solutions of the PDE for right-hand sides given by columns of $\dM$, defining $\dO_{:,c} \coloneqq \KM \dM_{:,c}$. 
    This matrix-vector products amounts to a single call of $\oracle$.
    Since the $\left\{\KM \de_i\right\}_{i \in c}$ have nearly disjoint supports after removing contributions from coarser scales (see \cref{fig:multicolor_cholesky}), this strategy allows gathering information contained in multiple solutions (corresponding to the number of elements of $c$) from just one application of $\oracle$ (one matrix-vector product). 
    
    \label{item:summary_measurement}
  \item We use \cref{alg:Cholesky} to construct a sparse approximate Cholesky factorization of $\KM$, the operator $\IK^{-1}$ (resp. $\IKM^{-1}$) represented in the basis given by $\{w_{i}\}_{i \in \I}$ (resp. $\{\dw_{i}\}_{i \in \I}$). \label{item:summary_recovery}
\end{enumerate}

\subsection{\ref{item:summary_multires}: Multiresolution basis} 
\label{sec:multires}
We begin by constructing a multiresolution basis using a Haar-type orthogonalization scheme, as illustrated in \cref{fig:basis_functions_haar,fig:basis_functions}. A detailed description of this construction can be found in  \cite[Sec.~5.3, 5.10]{owhadi2019operator}.

We assume that the computational domain $\Omega$ is represented by a quasi-uniform mesh with mesh-width $h_{\min}$.
Our method can be applied to a wide range of discretizations, including finite difference, finite elements, and finite volumes, as long as they are local. 
By this, we mean that the region of influence of each degree of freedom, such as the support of the associated basis function, has diameter $\lessapprox h_{\min}$.

\subsubsection*{Nested partition}
Assuming that the domain $\Omega \in \Reals^d$ has diameter one, we choose $1/2 \approx h \in (0,1)$ and $q \in \mathbb{N}_+$ such that $h^q \approx h_{\min}$.
For $1 \leq k \leq q$, we partition $\Omega$ into subsets $\tau^{(k)}_i$ (indexed by $i\in J^{(k)}$) of diameter $\leq h^k$, and that are nested, in the sense that each  $\tau^{(k)}_i$ is the union of elements $\tau^{(k+1)}_j$ of the finer partition. We write $\tau^{(k)}:=\{\tau^{(k)}_i\mid i\in J^{(k)}\}$.
We further let $\tau^{(0)} \coloneqq \{\Omega\}$ be the trivial partition.
As far as possible, we choose the $\tau^{(k)}_i$ to be convex with small aspect ratios (the aspect ratio of $\tau^{(k)}_i$ is defined as the ratio between the radius of the smallest ball containing that set and the radius of the largest ball contained in that set).

\begin{figure}
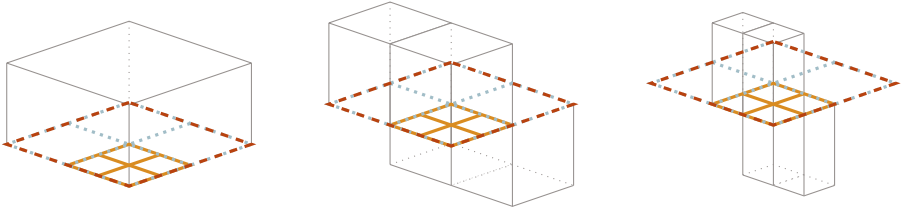

  \centering
  \begin{minipage}{0.32 \textwidth}
    \centering
    \tdplotsetmaincoords{70}{45}
    \begin{tikzpicture}[scale=1.15, tdplot_main_coords]
		  	\input{figures/tikz/basis_functions_haar_c.tex}
    \end{tikzpicture} 
  \end{minipage}
  \begin{minipage}{0.32 \textwidth}
    \centering
    \tdplotsetmaincoords{70}{45}
    \begin{tikzpicture}[scale=1.15, tdplot_main_coords]
		  	\input{figures/tikz/basis_functions_haar_m.tex}
    \end{tikzpicture} 
  \end{minipage}
  \begin{minipage}{0.32 \textwidth}
    \centering
    \tdplotsetmaincoords{70}{45}
    \begin{tikzpicture}[scale=1.15, tdplot_main_coords]
		  	\input{figures/tikz/basis_functions_haar_f.tex}
    \end{tikzpicture} 
  \end{minipage}
  \caption{Elements of {\color{rust} $\tau^{(1)}$}, {\color{lightblue} $\tau^{(2)}$}, and {\color{orange} $\tau^{(3)}$} and idealized basis functions $\cw$ on the first (left), second (center), and third (right) level of the hierarchy.}
  \label{fig:basis_functions_haar}
\end{figure}

\subsubsection*{Basis functions (continuum)} We begin with an idealized description of the multiresolution basis functions at the continuum level, under the assumption that we have access to the continuous solution map $\rhs \mapsto \IK^{-1} \rhs$.
For $t \in \tau^{(k)}$ an element of any of the $q$ partitions, we denote as $\cv_t$ the function that is $1$ on $t$ and zero everywhere else. 
We denote as $\cV^{(k)}$ the linear span of the functions in $\left\{\cv_{t}\right\}_{t \in \tau^(k)}$ (i.e., $\cV^{(k)}$ is the space of functions that are zero outside of $\Omega$ and (piecewise) constant in each $\tau^{(k)}_i$).
For each $1 \leq k \leq q$ we denote as $\cW^{(k)}$ the $L^2$-orthogonal complement of $\cV^{(k - 1)}$ in $\cV^{(k)}$, treating $\cV^{(0)}$ as $\{0\}$ (i.e., $\cW^{(k)}$ is the space of functions that are zero outside of $\Omega$, (piecewise) constant in each $\tau^{(k)}_i$, and of average zero in each $\tau^{(k-1)}_j$). 
We choose an $L^2$-orthonormal basis $(\cw_{i})_{i \in \I^{(k)}}$ for each $W^{(k)}$ that is local in the sense that the support of every $w_{i}$ is contained in (the closure) of an element of $\tau^{(k-1)}$. 
Abusing notation we write $\I \coloneqq \cup_{1 \leq k \leq q} \I^{(k)}$ and we consider the basis set $\left\{w_{i}\right\}_{i \in \I}$ as forming the columns of an infinitely long matrix denoted by $W^{(k)}$. 
An example of the resulting multiresolution basis is shown in \cref{fig:basis_functions_haar}.

\begin{figure}
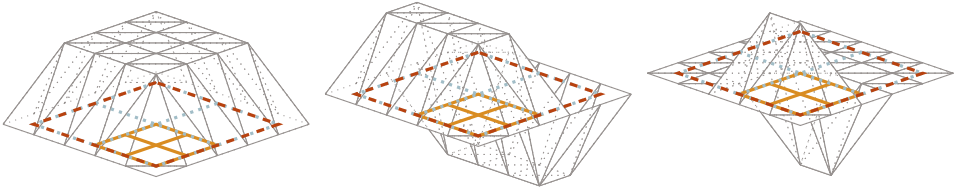

  \centering
  \begin{minipage}{0.32 \textwidth}
    \centering
    \tdplotsetmaincoords{70}{45}
    \begin{tikzpicture}[scale=1.15, tdplot_main_coords]
		  	\input{figures/tikz/basis_functions_c.tex}
    \end{tikzpicture} 
  \end{minipage}
  \begin{minipage}{0.32 \textwidth}
    \centering
    \tdplotsetmaincoords{70}{45}
    \begin{tikzpicture}[scale=1.15, tdplot_main_coords]
		  	\input{figures/tikz/basis_functions_m.tex}
    \end{tikzpicture} 
  \end{minipage}
  \begin{minipage}{0.32 \textwidth}
    \centering
    \tdplotsetmaincoords{70}{45}
    \begin{tikzpicture}[scale=1.15, tdplot_main_coords]
		  	\input{figures/tikz/basis_functions_f.tex}
    \end{tikzpicture} 
  \end{minipage}
  \caption{Elements of {\color{rust} $\tau^{(1)}$}, {\color{lightblue} $\tau^{(2)}$}, and {\color{orange} $\tau^{(3)}$} and basis functions $\dw$ on the first (left), second (center), and third (right) level of the hierarchy, on a piecewise affine finite element discretization.}
  \label{fig:basis_functions}
\end{figure}

\subsubsection*{Basis functions (discrete)}
We now describe how to adapt the above construction to working with the solver of a discretized elliptic PDE $\IKM$.
On a first reading, this somewhat technical section can be skipped.
Let $N$ be the number of rows and columns of $\IKM$. 
For each $1 \leq k \leq q$, $t \in \tau^{(k)}$, $i \in \I^{(k)}$, we need to construct discrete counterparts $\dv_{t}, \dw_{i} \in \Reals^N$ that mimic the behavior of $\cv_t, \cw_i$.
We outline a general approach, the specifics of which depend on the discretization. To illustrate this approach we assume that $\IKM$ is obtained by discretizing the continuous operator $\IK$ with a finite element method with elements $\psi_i$ indexed by $1 \leq i \leq N$, i.e. $\IKM_{i,j}=\int_\Omega \psi_i \IK^{-1} \psi_j$. We assume each $\psi_i$ to be supported in a subset of $\Omega$ of size $\mathcal{O}(h^q)$. We emphasize that the proposed approach does not require the precise knowledge of $\IKM$ or the finite elements $\psi_i$ used for its discretization, it only requires a rough identification of the location of the support of those elements.

For each $1 \leq k \leq q$, we construct a nested partition $\dtau^{(k)}$ of the $N$ discrete degrees of freedom that mirrors the partition $\tau^{(k)}$ of the domain.
We begin by assigning to each $t \in \tau^{(q)}$ the discrete counterpart $\dt(t) \coloneqq \{i\}$, where $i$ is chosen such that the support of $\psi_i$ is as close as possible to $t$, while ensuring that each $i$ is assigned to a unique $t \in \tau^{(q)}$. 
Working our way from finer to coarser levels, we then define a discrete counterpart $\dt(t)$ for each $t \in \tau^{(k)}$ and $1 \leq k < q$.
For each $1 \leq k < q$, and $t \in \tau^{(k - 1)}$ the union over elements of $\mathcal{N} \subset \tau^{(k)}$, we define $\dt(t) \coloneqq \cup_{s \in \mathcal{N}} \dt(s)$.
In other words, the discrete counterparts $\dt$ mimic the inclusion relationships among the associated $t$.
For $1 \leq k \leq q$, the $\dtau^{(k)} \coloneqq \{\dt(t)| t \in \tau^{(k)}\}$ then define the partitions of $\{1, \ldots N\}$.
For each $1 \leq k \leq q$ and $t \in \tau^{(k)}$, we approximate $\cv_t$ using degrees of freedom in the associated $\dt \in \dtau^{(k)}$. 
The resulting vector of coefficients is denoted as $\dv_t \in \Reals^N$ and the linear space spanned by the $(\dv_{t})_{t \in \tau^{(k)}}$ as $\dV^{(k)}$.

By constructing them from fine to coarse scales, we can ensure the $\dV^{(1)} \subset \dots \subset \dV^{(q)}$ are nested.
For $1 \leq k \leq q$, we then construct the $(\dw_{i})_{i \in I^{(k)}}$ as an orthonormal basis of the orthogonal complement of $\dV^{(k - 1)}$ in $\dV^{(k)}$ (where orthogonality and orthonormality are defined with respect to the Euclidean inner product), setting $\dV^{(0)} \coloneqq \{0\}$. 
As in the case of the $\cw_i$, we choose the $\dw_i$ to be \emph{local} in the sense that for $i \in \I^{(k)}$, the support of $\dw_i$ is restricted to a single $\dt \in \dtau^{(k - 1)}$.
Abusing notations we write $\I \coloneqq \cup_{1 \leq k \leq q} \I^{(k)}$ and denote as $\dW \in \Reals^{N \times \I}$ the orthonormal matrix that has the $(\dw_{i})_{i \in \I}$ as its columns, ordered from coarse to fine.
\cref{fig:basis_functions} illustrates this construction in the case of piecewise linear finite elements on a triangular mesh.

\begin{remark}
  If the order $2s$ of $\IK$ is larger than $d$, we can instead use a multiresolution subsampling scheme such as \cite[Example 5.1]{schafer2021compression}.
\end{remark}

\subsection{\ref{item:summary_coloring}: Coloring the basis functions}
\label{sec:coloring}
We now \emph{color} the basis functions, meaning that we create a partition $\colors$ of $\I$ such that any $i,j \in \I$ assigned to the same color $c \in \colors$ must be on the same level of the multiresolution basis and correspond to sufficiently distant basis functions.
To make this precise, we denote for $i \in \I^{(k)}$ as $t(\cw_i)$ the element of $\tau^{(k - 1)}$ that contains the support $\cw_i$.
Similarly, we define as $t(\dw_{i})$ the element of the partition $\tau^{(k - 1)}$ that is associated to the $\dt \in \dtau^{(k - 1)}$ that contains the support of $\dw_i$. 
We choose a tuning parameter $\rho > 0$ and color the set $\I$ such that
\begin{equation*}
  i \in \I^{(k)} \ \mathrm{and} \ j \in \I^{(l)} \ \text{of same color} \quad \Rightarrow \quad 
    k = l \ \mathrm{and} \, 
  \begin{cases} 
    \dist(t(\cw_i), t(\cw_j)) \geq 2 \rho h^{k - 1}\\
    \dist(t(\dw_i), t(\dw_j)) \geq 2 \rho h^{k - 1},   
  \end{cases}
\end{equation*}
where $\dist(\cdot, \cdot)$ is the ordinary Euclidean distance between subsets of $\Omega \subset \Reals^d$.
We can construct such a coloring by successively adding admissible elements of $\I$ to a given color until we run out of elements to add, adding new colors until all of $\I$ is colored.
We reorder the columns of $\dW$, such that those of the same color appear consecutively. 

\subsection{\ref{item:summary_measurement}: Taking measurements}
\label{sec:measurements}
So far, we have not used $\IK$, $\IKM$, or their inverses.
We use the coloring $\colors$ to decide, which solutions $\oracle(\rhs) \coloneqq \IK^{-1} \rhs$, $\oracle(\rhsM) \coloneqq \IKM^{-1} \rhsM$ to compute. 
We write $\de_i$ for the unit vector in the $i$-th direction, and denote as $\KM$ the matrix with entries $\KM_{ij} = \left \langle \cw_{i}, \IK^{-1} \cw_{j}\right \rangle_{L^2}$ (in the continuous case) or $\KM_{ij} = \dw_{i}^{\top} \IKM^{-1} \dw_{j}$ (in the discrete case).
We now define $\dM\in \Reals^{N \times \colors}$ and $\dO \in \Reals^{\I \times \colors}$ column-wise as 
\begin{equation*}
\dM_{:,c} \coloneqq \dm_{c} \coloneqq \sum_{i \in c} \de_i \quad \text{and} \quad \dO_{:,c} \coloneqq \disco_{c} \coloneqq \KM \dm_c.
\end{equation*}
Each matrix-vector product with $\KM$ requires a single call to the black box $\oracle$.
The matrix $\dO$ contains all numerical information about $\IKM$ or $\IK$ that we will use. 

\subsection{\ref{item:summary_recovery}: Computing the Cholesky factorization}
Recall that the elements of $\I$ were ordered from coarse to fine, with the elements of each color appearing consecutively. 
For $i,j \in \I$, we write $i \preceq j$ if $i$ appears before $j$ in this ordering.
We now construct an approximate Cholesky factorization of $\KM$. 
For a color $c \in \colors$ we define the operation $\texttt{scatter}_c(\solM)$ that takes in a vector $\solM \in \Reals^{\I}$ and splits it into a sparse matrix in $\Reals^{\I \times c}$ (replacing $w$ with $\dw$ in the discrete case),
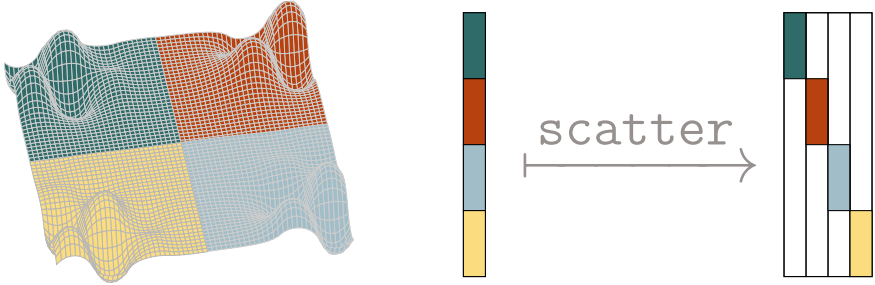
\begin{figure}
  \begin{minipage}{0.38 \textwidth}
    \centering
    \includegraphics[scale=0.675]{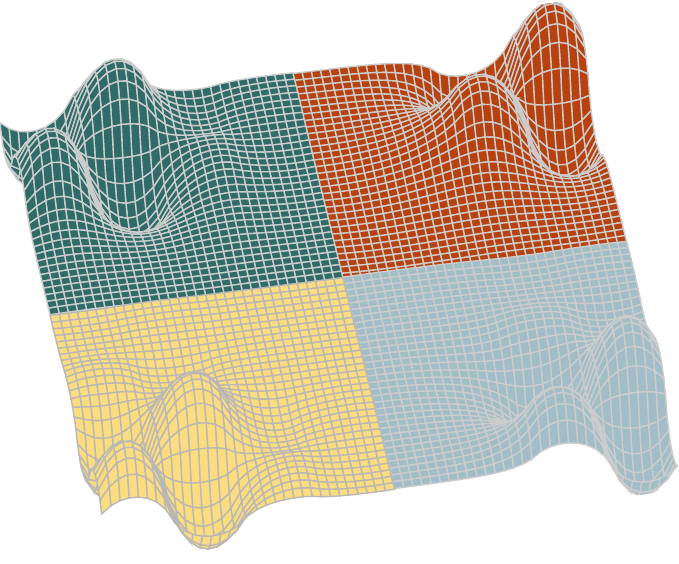}
  \end{minipage}
  \begin{minipage}{0.60 \textwidth}
    \centering
    \begin{tikzpicture}[scale=0.45]
		  	\begin{scope}
	\begin{scope}[xscale=1.3, yscale=1.3, xshift=-250]
		\draw (-0.5, 6) rectangle (0, 0);
 		\draw[fill=seagreen] (-0.5, 6.0) rectangle (0.0, 4.5);
 		\draw[fill=rust] (-0.5, 4.5) rectangle (0.0, 3.0);
 		\draw[fill=lightblue] (-0.5, 3.0) rectangle (-0.0, 1.5);
 		\draw[fill=joshua] (-0.5, 1.5) rectangle (-0.0, 0.0);
    \end{scope}

	\begin{scope}[xscale=1.3, yscale=1.3, xshift=-150]
        \node at (0.0, 3.0) {\Huge \textcolor{darksilver}{$\xmapsto{\mathtt{scatter}}$}};
    \end{scope}

	\begin{scope}[xscale=1.3, yscale=1.3]
		\draw (-2, 6) rectangle (0, 0);

%
 		\draw (-2.0, 6.0) rectangle (-1.5, 0.00);
 		\draw[fill=seagreen] (-2.0, 6.0) rectangle (-1.5, 4.5);
 
 		\draw (-1.5, 6.0) rectangle (-1.0, 0.00);
 		\draw[fill=rust] (-1.5, 4.5) rectangle (-1.0, 3.0);
 
 		\draw (-1.0, 6.0) rectangle (-0.5, 0.00);
 		\draw[fill=lightblue] (-1.0, 3.0) rectangle (-0.5, 1.5);
 
 		\draw (-0.5, 6.0) rectangle (-0.0, 0.00);
 		\draw[fill=joshua] (-0.5, 1.5) rectangle (-0.0, 0.0);
	\end{scope}
\end{scope}
    \end{tikzpicture} 
  \end{minipage}
  \caption{$\dO_{:, c} - \dL \dL^{\top} \dM_{:, c}$ is first obtained as a dense vector (left). 
  It is scattered into four sparse vectors containing the nonzeros corresponding to the $\dw_{i}$ closest to each of the $\{\dw_{\iota}\}_{\iota \in c}$ (right).} 
  \label{fig:scatter}
\end{figure}
\begin{equation*}
  \left(\texttt{scatter}_{c}\left(\solM\right)\right)_{ij} \coloneqq 
  \begin{cases}
    &\solM_{i}, \ \mathrm{if} \ j \preceq i \ \mathrm{and}  \ j = \argmin_{\iota \in c} \dist\left(t(w_i), t(w_\iota)\right)\\
    &0, \ \mathrm{else}.
  \end{cases}.
\end{equation*}
Here, we use an arbitrary method to break ties, ensuring that the $\argmin$ has a unique solution.
We write $\{\boldsymbol{\mu}_{j}\}_{j \in c}$ for the columns of the (sparse) matrix returned by $\texttt{scatter}_c$.
The method $\texttt{scatter}_{c}$ thus assigns the $i$-th entry of $\solM$ to the column $\boldsymbol{\mu}_{j}$ for which $\dw_{j}$ is closest to $\dw_{i}$.  
This is illustrated in \cref{fig:scatter}.
We write $\diag(\mathbf{B})$ for the diagonal of a (possibly non-square) matrix $\mathbf{B}$ and $\texttt{hcat}(\cdot, \cdot)$ for the function that takes two (sparse) matrices with equal number of rows and horizontally concatenates them.
With these definitions, the Cholesky factorization is computed by \cref{alg:Cholesky}.

\begin{algorithm}[H]
    \begin{algorithmic}[1]
        \STATE $\dL \leftarrow 0 \times \I$ empty matrix 
        \FOR{$c \in \colors$}
        \label{line:for_loop_cholesky}
        \STATE $\dL_{\mathrm{new}} \leftarrow \texttt{scatter}_c\left(  \dO_{:, c} - \dL \dL^{\top} \dM_{:, c}\right)$
        \STATE $\dL_{\mathrm{new}} \leftarrow \dL_{\mathrm{new}} \left(\diag\left(\dL_{\mathrm{new}}\right)\right)^{-1/2}$ 
        \STATE $\dL \leftarrow \texttt{hcat}\left(\dL, \dL_{\mathrm{new}}\right)$
        \ENDFOR
  			\RETURN $\dL$
	\end{algorithmic}
	\caption{\label{alg:Cholesky} Cholesky recovery}
\end{algorithm}

\begin{remark}
  Cholesky recovery can easily be extended to an $LU$ recovery that applies to nonsymmetric and indefinite PDEs, given access to matrix-transpose-vector products. 
  We conjecture that our approximation results are robust to perturbations of the PDE by possibly nonsymmetric or indefinite lower-order terms with bounded coefficients.
  Examples are advection-diffusion equations with bounded Pecl{\'e}t number and Helmholtz equations with bounded frequency.
\end{remark}

\newpage 
\section{Theoretical results}
\subsection{Setting for rigorous results}
\label{sec:rigorous_setting}
In a slightly more idealized setting, we can prove the accuracy of the factorization $\dL$ rigorously. 
We assume that $\Omega \subset \Reals^d$ is a Lipschitz-bounded domain and $\IK :H_0^s(\Omega) \longrightarrow H^{-s}(\Omega)$ is linear, bounded, invertible, self-adjoint, and local in the sense that $\int_\Omega u \IK v \dx = 0$ whenever $u,v \in H^s_0(\Omega)$ have disjoint support. 
We further assume that the $\{\tau^{(k)}\}_{1 \leq k \leq q}$ form nested partitions of $\Omega$ into convex and uniformly Lipschitz sets such that for $h,\delta \in (0, 1)$, each element of $\tau^{(k)}$ is contained in a ball of radius $h^k$ and contains a ball of radius $\delta h^k$. 
We then consider the idealized continuous multiresolution basis $\left\{w_{i}\right\}_{i \in \I}$ of \cref{sec:multires} and assume that we have access to matrix-vector multiplication with the resulting $\KM$. 
This setting corresponds to \cite[Example 2]{schafer2021compression}. Using the results therein, we prove rigorous results on the accuracy-vs-complexity tradeoff of our methods.
The requirement for the elements of the $\left\{\tau^{(k)}\right\}_{1 \leq k \leq q}$ to be convex arises only from the use Poincar{\'e} inequalities and Bramble-Hilbert lemmas for convex domains in \cite{schafer2021compression}.
Similar results could be obtained by using results for star-shaped domains or by assuming that the elements of $\left\{\tau^{(k)}\right\}_{1 \leq k \leq q}$ are obtained by intersecting $\Omega$ with a convex set. 

\subsection{Results for the simplicial case}
We begin by presenting a theoretical result for \cref{alg:Cholesky}, which we will refer to as \emph{simplicial} Cholesky recovery. 
In this setting, the accuracy-vs-complexity tradeoff of our method is characterized by the following theorem.

\begin{theorem}
  \label{thm:main_thm_simplicial}
  By choosing $\rho \gtrapprox \log(N / \epsilon)$, we    obtain a measurement matrix $\dM$ with $\O\big(\log(N) \log^{d}(N/\epsilon)\big)$ columns such that the $\dL$ produced by \cref{alg:Cholesky} applied to a suitable $\epsilon$-perturbation of $\KM$ satisfies $\dL \dL^{\top} = \KM$.
  By this, we mean that there exists a matrix $E \in \Reals^{N \times N}$ with $\|E\| \leq \epsilon$, such that the application of \cref{alg:Cholesky} to $\KM + E$ returns the Cholesky factor of $\KM$.
  It requires $\O\big(\log(N) \log^{d}(N/\epsilon)\big)$ matrix-vector products with $\KM$, as well as computational cost $\O\big(N \log^2(N) \log^{2d}(N/\epsilon)\big)$ and the result $\dL$ has $\O\big(N \log(N) \log^{d}(N/\epsilon)\big)$ nonzero entries.
  The hidden constants in $\gtrapprox, \O$ above depend on $d$, $\Omega$, $\|\IK\|$, $\|\IK^{-1}\|$, $\delta$, and $h$, but not on $q$, $N$, or $\epsilon$.
\end{theorem}
\begin{proof}
  The proof can be found in \cref{sec:proof_simplicial}.
\end{proof}

According to \cref{thm:main_thm_simplicial} there exists an exponentially small perturbation of the input, such that applying \cref{alg:Cholesky} to the \emph{perturbed} input yields the \emph{exact} solution. 
If we could prove that small perturbations of the input result in small perturbations of the output, this result would imply an accuracy estimate for \cref{alg:Cholesky}. 
Despite encouraging numerical results, we are not able to rigorously establish this result, just as in the case of incomplete Cholesky factorization \cite{schafer2021compression}. 
However, just as in \cite{schafer2021compression}, we can prove an end-to-end guarantee for a small variation of \cref{alg:Cholesky} operating on \emph{supernodes} instead of individual rows and columns of $\dL$. 

\subsection{Supernodal Cholesky recovery} 
The difficulty in proving accuracy results for \cref{alg:Cholesky} arises from the possible amplification of the truncation errors.
There will be on the order of $\log(N) \rho^{d}$ colors, and the truncation error is of order $\exp\left(-\gamma \rho\right)$ for a constant $\gamma$.
In principle, the error could be amplified by a constant factor $C > 1$ after peeling off each color, leading to a total error bound of the order of $\exp\left(\log(C) \log(N) \rho^d - \gamma \rho \right)$.
Since this bound is increasing in $\rho$, it is of no use in proving the accuracy of \cref{alg:Cholesky}, even though we do not observe such catastrophic error amplification in practice. 
In \cite{schafer2021compression}, this problem is addressed by aggregating nearby basis functions with similar sparsity patterns into supernodes with identical sparsity patterns.
The Cholesky factorization, and thus \cref{alg:Cholesky}, can then be performed on a block matrix defined with respect to this partition into supernodes.
Without significantly changing the sparsity pattern, we can choose supernodes of size $\approx \rho^d$.
If we color the supernodes just as we did in the simplicial case, we find that we only need $\O(\log(N))$ colors, resulting in a total error bound of $\exp\left(\log(C) \log(N) - \gamma \rho \right)$ for some $C > 1$, allowing us to rigorously prove the accuracy of our supernodal algorithm.
Supernodes also allow us to reformulate our algorithms in terms of a smaller number of level three BLAS operations.
In the closely related work of \cite{chen2021multiscale}, as well as sparse linear algebra more broadly \cite[Section 9]{davis2016survey}, supernodes are used to improve performance. 

\subsubsection{Supernodal aggregation}
A supernode is a set of indices that represents a group of basis functions, and thus of rows and columns of the matrices $\KM$ and $\dL$.
For our purposes, supernodes indexed by sets $\{\tilde{\I}^{(k)}\}_{1 \leq k \leq q}$ can be constructed as follows. 
\begin{enumerate}
  \item For each $1 \leq k \leq q$ select locations\footnote{For instance, this can be done by successively selecting points $x$ that violate the third condition.} $\left\{y_{\tilde{i}}\right\}_{\tilde{i} \in \tilde{\I}^{(k)}} \subset \Omega$ such that 
  \begin{align*}
    \min_{\tilde{i}, \tilde{j} \in \tilde{\I}^{(k)}} \dist\left(y_{\tilde{i}}, y_{\tilde{j}}\right) &\geq \rho h^k\\ 
    \min_{\tilde{i} \in \tilde{\I}^{(k)}} \dist\left(y_{\tilde{i}}, \partial \Omega\right) &\geq \rho h^k\\ 
    \max_{x \in \Omega} \min\left(\dist\left(x, \partial \Omega\right), \min_{\tilde{i} \in \tilde{\I}^{(k)}} \dist\left(y_{\tilde{i}}, x\right)\right) &\leq \rho h^k.
  \end{align*}
  \item Assign index $i \in \I^{(k)}$ to the supernode $\tilde{i}$ (writing $i \in \tilde{i}$) if $\tilde{i}$ is the element of $\tilde{\I}^{(k)}$ for which $y_{\tilde{i}}$ is closest to the centroid of $t(\dw_i)$, breaking ties arbitrarily.
\end{enumerate}

\subsubsection{Supernodal graph coloring}
Instead of coloring individual nodes in \cref{sec:coloring}, we now color entire supernodes such that for any supernodes $\tilde{i}$ and $\tilde{j}$ on scale $k$ with the same color, we have
\begin{equation*}
  i \in \tilde{i}, j \in \tilde{j} \Rightarrow \dist\left(t\left(\dw_i\right), t\left(\dw_j\right)\right) \geq 2 \rho h^{k - 1}. 
\end{equation*}
In practice, as in \cref{sec:coloring}, we try to choose the coloring, such as to reduce the number of colors needed, by packing supernodes of the same color as tightly as we can.
Crucially, the number of supernodal colors needed can be upper bounded by $\O(\log(N))$, with \emph{no dependence on} $\rho$.
As before, we denote as $\supercolors^{(k)}$ the supernodal colors on level $k$ and write $\supercolors \coloneqq \cup_{1 \leq k \leq q} \supercolors^{(k)}$.
We use $\preceq$ to denote the ordering of the supernodes from coarse to fine, with each color appearing consecutively.

\subsubsection{Supernodal measurements}
In \cref{sec:coloring}, each color resulted in a single column of the measurement matrix $\dM$. 
Here, we associate to each supernodal color $\tilde{c} \in \supercolors$ on level $k$ a supernodal measurement consisting of $m_{\tilde{c}} \coloneqq \max_{\tilde{i} \in \tilde{c}} \# \tilde{i}$ columns.
To define this measurement, we introduce an arbitrary ordering of the elements of each supernode and let the function $\eta: \I \longrightarrow \Naturals$ return the position of a basis function within this supernode. 
So, if $i$ is the first element of supernode $\tilde{i}$, we have $\eta(i) = 1$, and if $j$ is the third element of supernode $\tilde{j}$, we have $\eta(j) = 3$.
Denoting the $l$-th unit vector as $\mathbf{e}_{l}$, the supernodal measurement of $\tilde{c} \in \supercolors$ is then given by
\begin{equation}
\label{eqn:block-measurement}
  \dM_{:, \tilde{c}} = \sum \limits_{\tilde{i} \in \tilde{c}} \sum \limits_{i \in \tilde{i}} \de_{i} \de_{\eta(i)}^{\top} \subset \Reals^{N \times m_{\tilde{c}}}.
\end{equation}
In other words, we obtain each column of $\dM_{:, \tilde{c}}$ as in \cref{sec:measurements} after picking one basis function from each supernode in $\tilde{c}$, picking each basis function only once. 

\subsubsection{Computing the supernodal Cholesky factorization}
The supernodal measurements $\dM_{:, \tilde{c}}$ for each color $\tilde{c} \in \supercolors$, allow us to use a blocked version of the algorithm that performs key operations using dense level-three BLAS. 
To this end, we consider the matrices $\dO$, and $\dL$ as $\tilde{\I} \times \tilde{\I}$ block-matrices, the $(\tilde{i}, \tilde{j})$-th entry of which is a matrix in $\Reals^{\tilde{i}, \tilde{j}}$.
The matrix $\dM$ is treated as a $\tilde{I} \times \supercolors$ matrix, with the $(\tilde{i}, \tilde{c})$ entry a matrix in $\Reals^{\tilde{i} \times m_{\tilde{c}}}$ defined by \cref{eqn:block-measurement}.
We define a block-version of scatter taking in a block-vector $\solM \in \Reals^{\tilde{I} \times m_{\tilde{c}}}$ and splits it into a sparse matrix in $\Reals^{\tilde{I} \times \tilde{c}}$ defined by:
\begin{equation*} 
  \left(\texttt{scatter}_{\tilde{c}} \left(\solM\right) \right)_{\tilde{i}\tilde{j}} \coloneqq 
  \begin{cases}
    &\solM_{\tilde{i}, 1 : \# \tilde{j} }, \ \mathrm{if} \ \tilde{j} \prec \tilde{i} \ \mathrm{and}  \ \tilde{j} = \argmin_{\tilde{\iota} \in \tilde{c}} \dist\left(y_{\tilde{i}}, y_{\tilde{\iota}})\right)\\
    &0, \ \mathrm{else}
  \end{cases},
\end{equation*} 
where we use an arbitrary method to break ties, ensuring that the $\argmin$ has a unique solution.
Letting $\diag\left(\dL\right)$ extract the block-diagonal part of $\dL$ and denoting as $\diag(\dL)^{-1/2}$ the inverse-Cholesky factorization of the block-diagonal part, the resulting \cref{alg:SuperCholesky} is almost identical to \cref{alg:Cholesky}, just acting on supernodal blocks instead of individual entries.

\begin{algorithm}[H]
    \begin{algorithmic}[1]
        \STATE $\dL \leftarrow 0 \times \tilde{\I}$ empty matrix 
        \FOR{$\tilde{c} \in \supercolors$}
        \STATE $\dL_{\mathrm{new}} \leftarrow \texttt{scatter}_{\tilde{c}}\left( \dO_{:, \tilde{c}} - \dL \dL^{\top} \dM_{:, \tilde{c}}\right)$
        \STATE $\diag(\dL_{\mathrm{new}}) \leftarrow \frac{\diag\left(\dL_{\mathrm{new}}\right) + \diag\left(\dL_{\mathrm{new}}\right)^{\top}}{2}$
        \STATE $\dL_{\mathrm{new}} \leftarrow \dL_{\mathrm{new}} \diag\left(\dL_{\mathrm{new}}\right)^{-1/2}$
 
        \STATE $\dL \leftarrow \texttt{hcat}\left(\dL, \dL_{\mathrm{new}}\right)$
        \ENDFOR
  			\RETURN $\dL \diag\left(\dL\right)^{-1/2}$
	\end{algorithmic}
	\caption{\label{alg:SuperCholesky} Supernodal Cholesky recovery}
\end{algorithm}

\subsubsection{Results for the supernodal case}
When using the supernodal multicolor ordering, the total number of colors is upper bounded by $\O\left(\log(N)\right)$, which allows us to rigorously control error propagation throughout the Cholesky recovery.

\begin{theorem}
  \label{thm:main_thm_supernodal}
  By choosing $\rho \gtrapprox \log(N / \epsilon)$, we obtain a measurement matrix $\dM$ with $\O\big(\log(N) \log^{d}(N/\epsilon)\big)$ columns such that the $\dL$ produced by \cref{alg:SuperCholesky} satisfies $\big\| \dL \dL^{\top} - \KM \big\| \leq \epsilon$.
  This requires $\O\big(\log(N) \log^{d}(N/\epsilon)\big)$ matrix-vector products with $\KM$ and additional computational cost $\O\big(N \log^2(N) \log^{2d}(N/\epsilon)\big)$. The resulting $\dL$ has $\O\big(N \log(N) \log^{d}(N/\epsilon)\big)$ nonzero entries.
  The hidden constants in $\gtrapprox, \O$ above depend on $d$, $\Omega$, $\|\IK\|$, $\|\IK^{-1}\|$, $\delta$, and $h$, but not on $q$, $N$, or $\epsilon$.
\end{theorem}
\begin{proof}
  The proof can be found in \cref{sec:proof_supernodal}.
\end{proof}
By adopting the supernodal version of the algorithm, we provide a rigorous end-to-end guarantee on the approximation accuracy of our method and its relationship to the computational cost and the number of matrix-vector products with $\KM$ required.

\begin{remark}
  Analog to \cite{schafer2021compression}, one can slightly adapt the (simplicial) sparsity pattern, which in our case is given by the simplicial $\texttt{scatter}$ operation, to obtain the same end-to-end guarantees for \cref{alg:Cholesky} applied in the supernodal ordering.
\end{remark}

\subsection{Low-rank approximation}
In the above, we have considered the improvement of the approximation accuracy as we increase the tuning parameter $\rho$ and thus the size of the sparsity pattern of $\dL$.
Instead, we could choose $\rho = \infty$ and limit the number of columns of $\dL$, obtaining a low-rank approximation of $\KM$.
Following \cite{schafer2021compression}, this approach results in a \emph{near-optimal} low-rank approximation, matching the $2$-norm approximation of the eigendecomposition up to a constant factor.

\begin{theorem}
  \label{thm:low_rank}
  In the setting of \cref{thm:main_thm_simplicial,thm:main_thm_supernodal}, let $\dL$ be the exact Cholesky factor of $\KM$ (which is obtained from setting $\rho = \infty$ and using \cref{alg:Cholesky,alg:SuperCholesky}).
  Denoting by $\left\|\cdot\right\|$ the matrix-$2$-norm, we have
  \begin{equation*}
    \left\| \KM - \dL_{:, 1:k} \left(\dL_{:, 1:k}\right)^{\top} \right\| \leq C \left\|\KM \right\| k ^{-2s / d}
  \end{equation*}
  with the constant $C$ depending only on $\|\IK\|$, $\|\IK^{-1}\|$, $d$, $h$, $\delta$, and $s$. 
  This is, up to a constant, the optimal rate decay of any low-rank approximation.
\end{theorem}
\begin{proof}
  The proof can be found in \cref{sec:proof_low_rank}.
\end{proof}

Asymptotically speaking, the accuracy-vs-cost (both in terms of computation and matrix-vector products with $\KM$) of the sparse approximation obtained by setting $\rho < \infty$ clearly outperforms that of the low-rank approximation obtained by setting $k < N$.  
In practice, it can be beneficial to utilize \emph{both} approximations by restricting \cref{alg:Cholesky,alg:SuperCholesky} to only some of the colors. 
This can improve the accuracy for a given computational budget and provides an approximation of the leading eigenspace.

\subsection{Approximation of the continuous Green's function}
Our theoretical results on the approximation of the discretized Green's function can be extended to approximations of the continuous Green's function $\K \in L^{2}\left(\Omega \times \Omega\right)$, by piecewise constant or piecewise affine interpolation. 

\subsubsection{Approximation in operator norm}
From the approximate factorizaton $\KM \approx \dL \dL^{\top}$ computed using \cref{alg:SuperCholesky} with a given $\rho$, we define the associated approximate continuous Green's function as 
\begin{equation*}
  \bar{\K}_{\rho}\left(x, y\right) = \sum_{i \in \cup_{k \leq q} \I^{(k)}} \sum_{j \in \cup_{k \leq q} \I^{(k)}} \left(\dL \dL^{\top}\right)_{ij} w_{i}(x) w_{j}(y).
\end{equation*}
We denote $\bar{\K}_{\infty}$ simply as $\bar{\K}$ and observe that it is equal to 
\begin{equation*} 
  \bar{G}(x, y) = \sum_{i \in \cup_{k \leq q} \I^{(k)}} \sum_{j \in \cup_{k \leq q} \I^{(k)}} \left\langle w_{i}, \K w_{j} \right\rangle_{L^2} w_{i}(x) w_{j}(y).
\end{equation*} 
Like the Green's function $\K$, the projected Green's function $\bar{\K}$ defines a map from $L^2\left(\Omega\right)$ to itself by $\rhs \mapsto \int_{\Omega} \bar{\K}(\cdot, y) \rhs(y) \mathrm{d}y$.
We denote the $L^2$ operator norm as $\|\cdot\|_{L^2 \rightarrow L^2}$.
\begin{theorem}
  \label{thm:continuous_operator}
  In the setting of \cref{thm:main_thm_supernodal}, for $\rho \gtrapprox \log(N)$ and $q$ chosen such that $h^q \approx N^{-1/d}$, we have for a constant $C$ depending only on $d, \Omega, \IK, h, \delta, p$,
  \begin{equation*}
   \|\bar{\K}_{\rho} - \K \|_{L^2 \rightarrow L^2} \leq CN^{-1 / d}.
  \end{equation*}
  Thus, $N \approx \epsilon^{-d}$ yields an $\epsilon$ approximation of $\KM$ in $\|\cdot\|_{L^2 \rightarrow L^2}$ norm from $C\left(\log^{d + 1}(\epsilon)\right)$ matrix-vector products with $\KM$, and thus applications of the solution operator $\K$.
\end{theorem}

\subsection{Approximation in Hilbert-Schmidt norm}
We can also obtain estimates on the Green's function in $L^2\left(\Omega \times \Omega\right)$ or, equivalently, estimates of the associated operator in the Hilbert-Schmidt norm.
This is, in general, not possible with a piecewise constant Green's function, but requires projecting the approximate Green's function on a space of piecewise polynomials. 
For a nested partition $\tau$ with $\hat{q}$ levels and a positive integer $p$, we define the doubly projected Green's function $\hat{\bar{\K}}$ as 
\begin{equation*}
  \hat{\bar{G}}_{\rho, \hat{q}} (x, y) = \sum_{\hat{i} \in \cup_{k \leq \hat{q}} \hat{\I}^{(k)}} \sum_{\hat{j} \in \cup_{k \leq \hat{q}} \hat{\I}^{(k)}} \left\langle \hat{w}_{\hat{i}}, \bar{\K}_{\rho} \hat{w}_{\hat{j}} \right\rangle_{L^2} \hat{w}_{\hat{i}}(x) \hat{w}_{\hat{j}}(y),
\end{equation*}
where the $\left\{\hat{w}_{i}\right\}_{i \in \cup_{k \leq \hat{q}} \hat{\I}^{(k)}}$ are a local orthogonal multiresolution basis of functions that are polynomials of order $p - 1$  on elements of $\tau^{(\hat{q})}$ in the same way in which the $\left\{w_{i}\right\}_{i \in \cup_{k \leq q} \I^{(k)}}$ form a local orthogonal basis of functions that are constant on each element of $\tau^{(q)}$. 
We again write $\hat{\bar{\K}}_{\infty, \hat{q}} = \hat{\bar{K}}_{\hat{q}}$ and will furthermore drop the explicit dependence on $\hat{q}$, where it is unambiguous.
For the important case of $s = 1, d \in \{2,3\}$, we also need higher-order regularity. 
For instance, in the case of second-order divergence form elliptic PDE operators $\IK$ with $C^1$-coefficients and $C^2$ boundary, it can be shown that they form bounded invertible maps not just from $H^{-1}$ to $H^1_{0}$ but also from $L^2$ to $H^{2} \cap H_{0}^1$.

\begin{theorem}
  \label{thm:continuous_l2}
  In the setting of \cref{thm:main_thm_supernodal}, assume that the nested partition can be convexly refined infinitely to $\left\{\tau^{(k)}\right\}_{1 \leq k < \infty}$ for a given regularity parameter $\delta$. 
  Assume also that for $s \leq r \leq 2s$, $\IK$ is a bounded invertible map from $H^{r - 2s}\left(\Omega\right)$ to $\left(H_{0}^s \cap H^{r}\right)\left(\Omega\right)$.
  Choose $q$ such that $h^q \approx N^{-1/d}$.\\
  Then, for $C$ depending only on $d, \Omega, \IK, h, \delta, p$ and $\rho > C\log(N)$, $\min(p, r) > d/2$, and $\hat{q}$ such that $\lceil q / \min(p, r) \rceil \geq \hat{q} \geq \lfloor q / \min(p, r) \rfloor$, we have
  \begin{equation*}
   \|\hat{\bar{\K}}_{\rho, \hat{q}} - \K \|_{L^2 \otimes L^2} \leq CN^{-\left(1 - \frac{d}{2\min(p, r)}\right)}.
  \end{equation*}
  By choosing $N$ according to $\epsilon \approx N^{-\left(1 - \frac{d}{2\min(p, r)}\right)}$, we thus obtain an $\epsilon$ approximation of $\K$ in $\|\cdot\|_{L^2 \otimes L^2}$ norm from $C\left(\log^{d + 1}(\epsilon)\right)$ applications of the solution operator $\K$.
\end{theorem}

As a byproduct of our proof, we also obtain the following corollary.

\begin{corollary}[\cref{cor:hsnorm} of the appendix]
  \label{cor:hsnorm_main}
  In the setting of \cref{thm:continuous_l2}, we can reconstruct the Green's function $\K$ to accuracy $\epsilon$ from solutions of the PDE for $C \epsilon^{\frac{2}{2\min(p, r) / d - 1}}$ that are piecewise polynomials (of order $p$) right-hand sides.
  This is possible without using sparsity, using only a global $L^2$-projection. 
\end{corollary}

\begin{remark}
  By \cref{cor:hsnorm_main}, the approximation rates of \cite{boulle2021learning} can be improved by identifying $\left(\K w_{i}\right)_{1 \leq i \leq m}$ one by one, for $\left(w_i\right)_{1 \leq i \leq m}$ a basis of local polynomials.  
\end{remark}

\section{Numerical Experiments}
\label{sec:numerics}

\subsection{Setup}
We now present numerical evidence for the practical utility of our method, focusing on the tradeoff between the number of matrix-vector products and the accuracy of the resulting approximation. 
We construct the multiresolution basis as in \cref{sec:multires}, using the collocation points of the finite difference and finite element discretizations to compute distances between basis functions. 
We construct the graph-coloring by successively selecting the basis function that is furthest from the basis functions that were already colored in a given color, using a fast nearest-neighbor lookup and a binary heap to achieve near-linear complexity.
We follow the implementation of \cref{alg:Cholesky} with the only difference that we use the original basis $\left\{e_{i}\right\}_{i \in \hat{\I}}$ given by the finite difference points to represent the column space of $\dL$, instead of the multiresolution basis given by the $\left\{\dw_{i}\right\}_{i \in \I}$. 
The code for reproducing the experiments can be found under \url{https://github.com/f-t-s/sparse_recovery_of_elliptic_solution_operators_from_matrix-vector_products}.

\begin{figure}
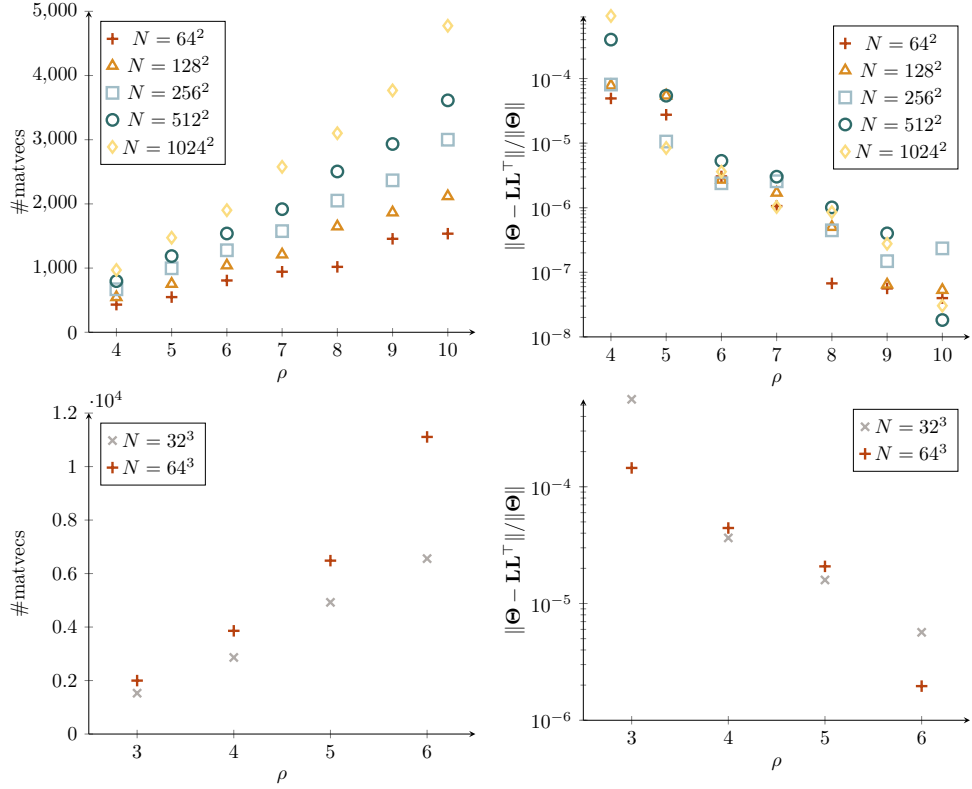

  \centering
  \begin{minipage}{0.495 \textwidth}
    \centering
    \begin{tikzpicture}[scale=0.75]
		  	\begin{axis}[
    axis lines = left,
    xlabel = \(\rho\),
    ylabel = \(\# \mathrm{matvecs}\),
    legend pos = north west,
    xmin = 3.5, xmax = 10.5,
    ymin = 0, ymax = 5000,
]
\addplot+[
    only marks,
    very thick,
    mark=+,
    color=rust,
    mark size=2.9pt]
table[col sep=comma,x index=0,y index=1]{./figures/csv/rel_err_plots/poisson/64_2_periodic_constant_random_mild.csv};
\addlegendentry{\(N = 64^2\)}

\addplot+[
    only marks,
    very thick,
    mark=triangle,
    color=orange,
    mark size=2.9pt]
table[col sep=comma,x index=0,y index=1]{./figures/csv/rel_err_plots/poisson/128_2_periodic_constant_random_mild.csv};
\addlegendentry{\(N = 128^2\)}

\addplot+[
    only marks,
    very thick,
    mark=square,
    color=lightblue,
    mark size=2.9pt]
table[col sep=comma,x index=0,y index=1]{./figures/csv/rel_err_plots/poisson/256_2_periodic_constant_random_mild.csv};
\addlegendentry{\(N = 256^2\)}

\addplot+[
    only marks,
    very thick,
    color=seagreen,
    mark=o,
    mark size=2.9pt]
table[col sep=comma,x index=0,y index=1]{./figures/csv/rel_err_plots/poisson/512_2_periodic_constant_random_mild.csv};
\addlegendentry{\(N = 512^2\)}

\addplot+[
    only marks,
    very thick,
    color=joshua,
    mark=diamond,
    mark size=2.9pt]
table[col sep=comma,x index=0,y index=1]{./figures/csv/rel_err_plots/poisson/1024_2_periodic_constant_random_mild.csv};
\addlegendentry{\(N = 1024^2\)}

\end{axis}
    \end{tikzpicture} 
    \begin{tikzpicture}[scale=0.75]
		  	\begin{axis}[
    axis lines = left,
    xlabel = \(\rho\),
    ylabel = \(\# \mathrm{matvecs}\),
    legend pos = north west,
    xmin = 2.5, xmax = 6.5,
    ymin = 0, ymax = 12000,
]

\addplot+[
    only marks,
    very thick,
    mark=x,
    color=silver,
    mark size=2.9pt]
table[col sep=comma,x index=0,y index=1]{./figures/csv/rel_err_plots/poisson/32_3_periodic_constant_random_mild.csv};
\addlegendentry{\(N = 32^3\)}

\addplot+[
    only marks,
    very thick,
    mark=+,
    color=rust,
    mark size=2.9pt]
table[col sep=comma,x index=0,y index=1]{./figures/csv/rel_err_plots/poisson/64_3_periodic_constant_random_mild.csv};
\addlegendentry{\(N = 64^3\)}


\end{axis}
    \end{tikzpicture} 
  \end{minipage}
  \begin{minipage}{0.495 \textwidth}
    \centering
    \begin{tikzpicture}[scale=0.75]
		  	\input{figures/tikz/rel_err_plot_periodic_mild_2d.tex}
    \end{tikzpicture} 
    \begin{tikzpicture}[scale=0.75]
		  	\input{figures/tikz/rel_err_plot_periodic_mild_3d.tex}
    \end{tikzpicture} 
  \end{minipage}
    \caption{We plot the number of matrix-vector products (left) and resulting relative error in operator-$2$-norm (right) of our method for a constant coefficient Laplacian with a random potential and periodic boundary conditions, as discussed in \cref{sec:laplace_mild}. 
    The top row shows results in $[0, 1)^2$, and the bottom row in $[0, 1)^3$.}
    \label{fig:matvecplots_mild}
\end{figure}

\begin{figure}
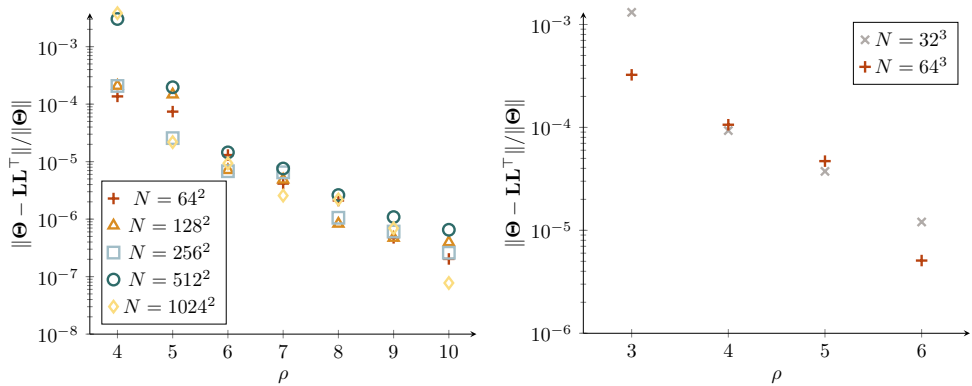

  \centering
  \begin{minipage}{0.495 \textwidth}
    \centering
    \begin{tikzpicture}[scale=0.75]
		  	\input{figures/tikz/rel_err_plot_periodic_severe_2d.tex}
    \end{tikzpicture} 
  \end{minipage}
  \begin{minipage}{0.495 \textwidth}
    \centering
    \begin{tikzpicture}[scale=0.75]
		  	\input{figures/tikz/rel_err_plot_periodic_severe_3d.tex}
    \end{tikzpicture} 
  \end{minipage}
    \caption{We plot the approximation rates of our method in the setting of rough high-contrast coefficients discussed in \cref{sec:laplace_mild}, showing that its accuracy decreases only slightly compared to \cref{fig:matvecplots_mild}. The \#matvecs per  $\rho$ are the same as in \cref{fig:matvecplots_mild}.}
    \label{fig:matvecplots_severe}
\end{figure}

\subsection{Laplacian with random potential}
\label{sec:laplace_mild}
Following \cite{lin2011fast}, we consider a Laplace operator over $[0, 1)^2$ with periodic boundary conditions, discretized over a regular grid of $N = n^2$ points.\footnote{Note that the $N$ of \cite{lin2011fast} is equivalent to the $n$ in the present work.}
Again following \cite{lin2011fast}, we add to the $i$-th diagonal entry of the matrix representation of $-\Delta$ a zeroth order term $1 + W_i$, where the $W_i \sim \operatorname{UNIF}([0, 1])$ are independent.
The zeroth-order term removes the null space of the differential operator and, by its roughness, prevents higher-order regularity of the PDE.
This allows showcasing the fact that our method, just as those of \cite{lin2011fast}, does not rely on higher-order smoothness of solutions of the PDE.
In \cref{fig:matvecplots_mild}, we have plotted the number of matrix-vector-products and the resulting relative error of our approximation as a function of $\rho$.
For the same relative accuracy of about $10^{-7}$, our method requires about half as many matrix-vector products compared to the results reported in \cite{lin2011fast}.
To demonstrate that our method is not restricted to two-dimensional problems, we provide results for the same setting, but with $N = n^3$ points in $[0, 1)^3$.

\begin{figure}
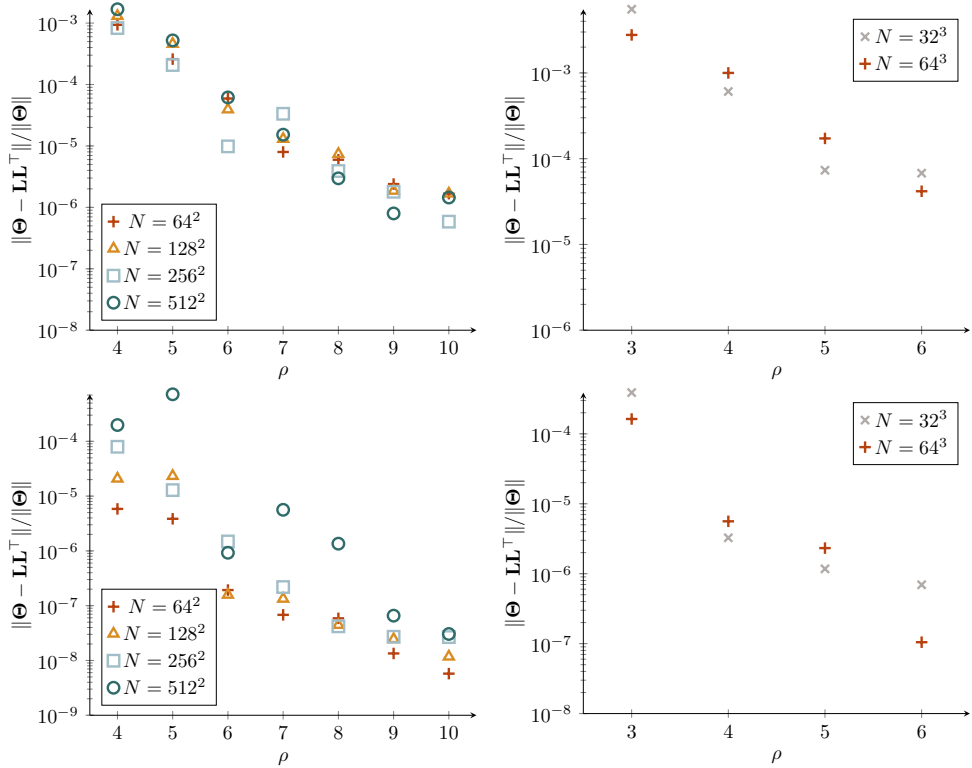

  \centering
  \begin{minipage}{0.495 \textwidth}
    \centering
    \begin{tikzpicture}[scale=0.75]
		  	\input{figures/tikz/rel_err_plot_fractional_12_2d.tex}
    \end{tikzpicture} 
    \begin{tikzpicture}[scale=0.75]
		  	\input{figures/tikz/rel_err_plot_fractional_32_2d.tex}
    \end{tikzpicture} 
  \end{minipage}
  \begin{minipage}{0.495 \textwidth}
    \centering
    \begin{tikzpicture}[scale=0.75]
		  	\input{figures/tikz/rel_err_plot_fractional_12_3d.tex}
    \end{tikzpicture} 
    \begin{tikzpicture}[scale=0.75]
		  	\input{figures/tikz/rel_err_plot_fractional_32_3d.tex}
    \end{tikzpicture} 
  \end{minipage}
    \caption{We apply our method to solution operators of the fractional Laplace equation described in \cref{sec:laplace_fractional}. 
    The first row shows the case $s = 0.5$, and the second row the case $s = 1.5$. 
    The first column shows the case $d = 2$ and the second column the case $d = 3$.}
    \label{fig:matvecplots_fractional}
\end{figure}

\subsection{Laplacian with rough coefficients}
\label{sec:laplace_severe}
To show that our methods can work under minimal regularity assumptions, we also provide results with rough coefficients. 
In the setting of \cref{sec:laplace_mild}, we obtain each conductivity coefficient as $Z_{i} + 10^{-4}$, where the $Z_i \sim \operatorname{UNIF}([0, 1])$ are independent.
The resulting problem has a contrast ratio of about $10^4$. 
As shown in \cref{fig:matvecplots_severe} the approximation rates obtained by our method are hardly affected, showing its robustness beyond the scope of our rigorous results (which are not robust to high contrast coefficients).

\subsection{Fractional Laplace operators}
\label{sec:laplace_fractional}
The exponential decay of Cholesky factors that forms the basis of our method was observed to hold even in for Mat{\'e}rn kernels corresponding to fractional-order PDEs \cite{schafer2021compression}.
Motivated by this observation, we now attempt to recover the solution operators of fractional Laplacian equations. 
We use the spectral definition of the fractional Laplacian together with the fast Fourier transform efficiently evaluate the solution operator. 
In \cref{fig:matvecplots_fractional}, we show the results of our method applied to the solution operators of the equation $u \mapsto (-\Delta)^s u  + u = f$ for $s \in \{0.5, 1.5\}$. 
Even though fractional Laplace operators are nonlocal, discretizations of their inverses have near-local Cholesky factors.

\begin{figure}
  \includegraphics[width=0.49\textwidth]{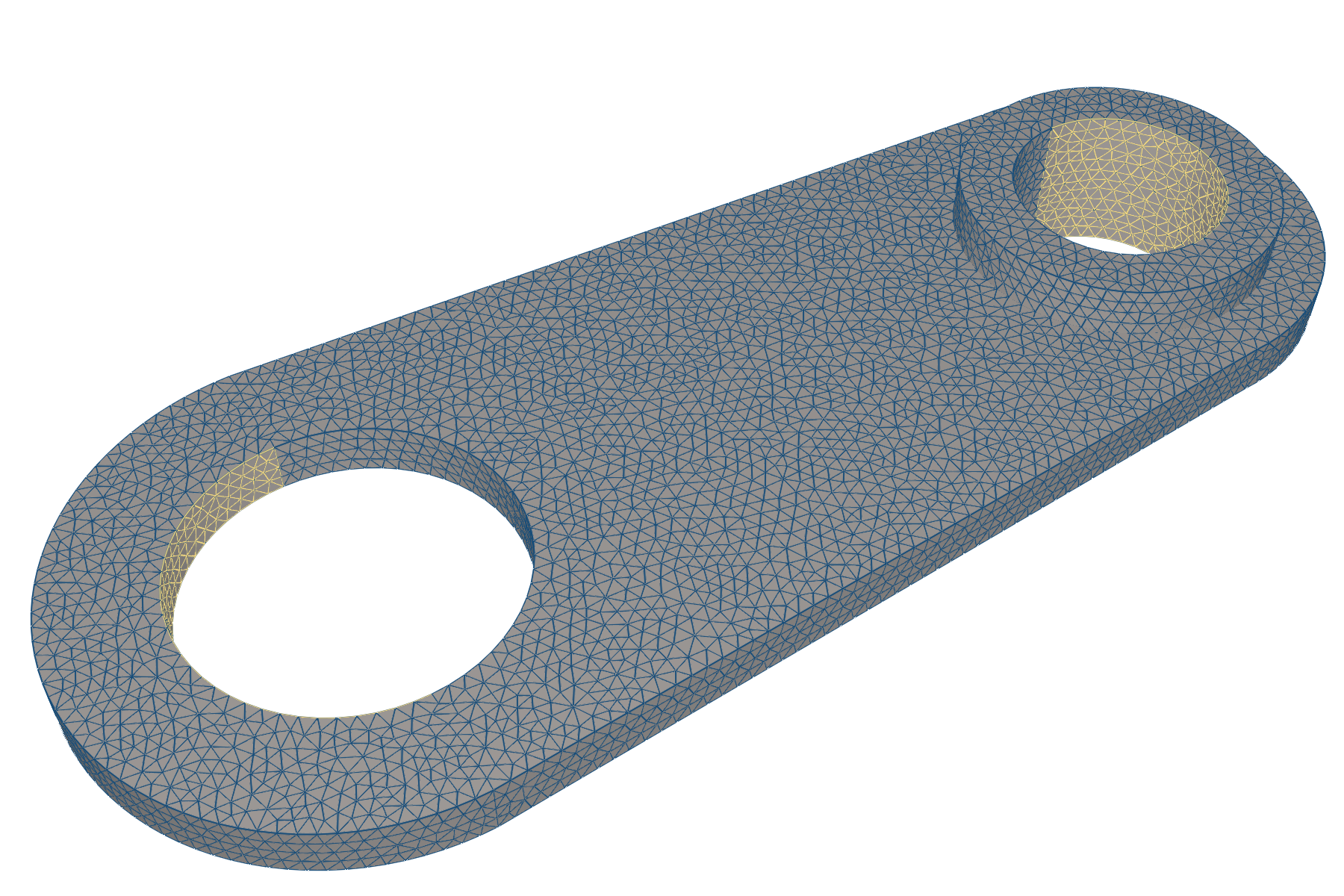}
  \includegraphics[width=0.49\textwidth]{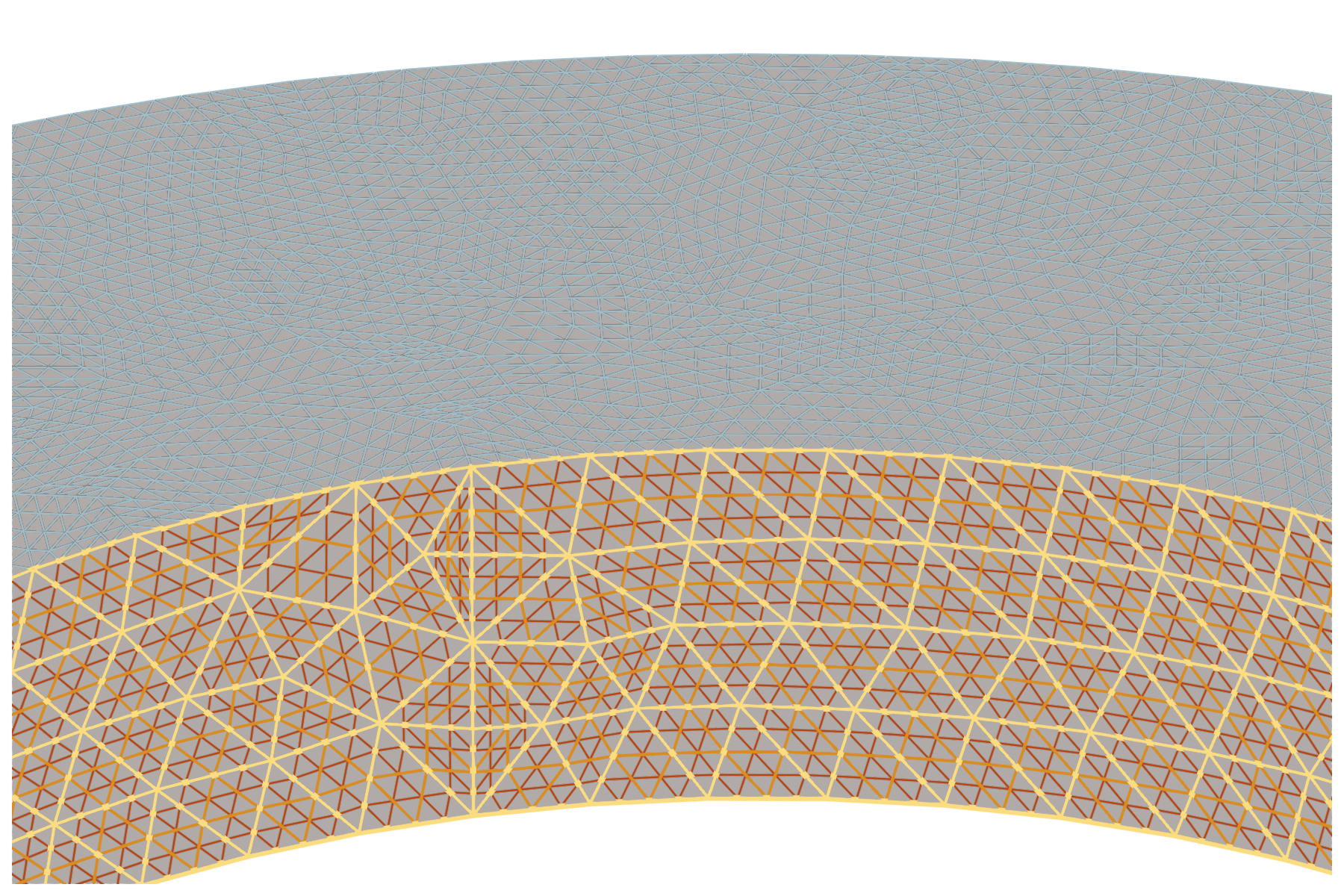}
  \caption{Left: The (volumetric) finite element model used in \cref{sec:laplace_gridap} at its coarsest resolution, Dirichlet boundary conditions (BC) marked in yellow, and Neumann BC in blue. Right: The three different levels of refinement. 
  Model and discretization provided by \cite{Badia2020}.}
  \label{fig:gridap_model}
\end{figure}

\begin{figure}
  \centering
  \begin{minipage}{0.495 \textwidth}
    \centering
    \begin{tikzpicture}[scale=0.75]
		  	\begin{axis}[
    axis lines = left,
    xlabel = \(\rho\),
    ylabel = \(\# \mathrm{matvecs}\),
    legend pos = north west,
    xmin = 2.5, xmax = 8.5,
    ymin = 0, ymax = 18000,
]

\addplot+[
    only marks,
    very thick,
    mark=x,
    color=silver,
    mark size=2.9pt]
table[col sep=comma,x index=0,y index=1]{./figures/csv/rel_err_plots/gridap/1.csv};
\addlegendentry{\(N = 9925\)}

\addplot+[
    only marks,
    very thick,
    mark=+,
    color=rust,
    mark size=2.9pt]
table[col sep=comma,x index=0,y index=1]{./figures/csv/rel_err_plots/gridap/2.csv};
\addlegendentry{\(N = 66111\)}


\addplot+[
    only marks,
    very thick,
    mark=triangle,
    color=orange,
    mark size=2.9pt]
table[col sep=comma,x index=0,y index=1]{./figures/csv/rel_err_plots/gridap/3.csv};
\addlegendentry{\(N = 478953\)}

\end{axis}
    \end{tikzpicture} 
  \end{minipage}
  \begin{minipage}{0.495 \textwidth}
    \centering
    \begin{tikzpicture}[scale=0.75]
		  	\input{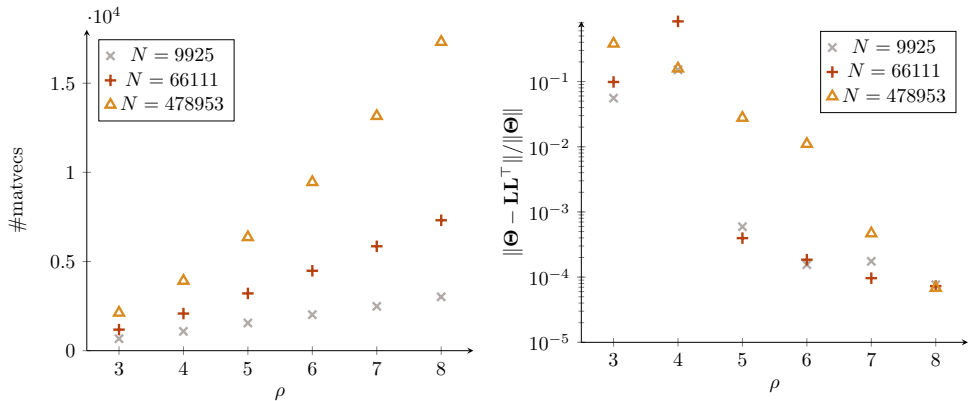}
    \end{tikzpicture} 
  \end{minipage}
    \caption{We plot the number of matrix-vector products (left) and resulting relative error in operator-$2$-norm (right) of our method for the finite-element discretization shown in \cref{fig:gridap_model}, as discussed in \cref{sec:laplace_gridap}.}
    \label{fig:matvecplots_gridap}
\end{figure}

\subsection{Finite element discretization} 
\label{sec:laplace_gridap}
To show that our method can handle finite element discretizations, mixed boundary conditions, and complex geometries, we apply our method to a Laplace operator discretized on the finite element model shown in \cref{fig:gridap_model}, with Dirichlet and Neumann boundary conditions.
We consider three levels of refinement with $N = \left\{9925, 66111, 478953\right\}$ degrees of freedom. 
\Cref{fig:gridap_model} shows the exponential decay of the approximation error as a function of $\rho$.

\section*{Acknowledgments}
This research was supported in part through research cyberinfrastructure resources and services provided by the Partnership for an Advanced Computing Environment (PACE) at the Georgia Institute of Technology, Atlanta, Georgia, USA.
Both authors gratefully acknowledge support from the Air Force Office of Scientific Research under award number FA9550-20-1-0358 (Machine Learning and Physics-Based Modeling and Simulation) and award number FA9550-18-1-0271 (Games for Computation and Learning) and from the Office of Naval Research under award number N00014-18-1-2363 (Toward scalable universal solvers for linear systems).
FS gratefully acknowledges support from the Office of Naval Research under award number N00014-23-1-2545 (Untangling Computation). We thank the anonymous referees for helping us improve the presentation of our results. 

\bibliographystyle{siamplain}
\bibliography{references}

\begin{thebibliography}{10}

\bibitem{altmann2021numerical}
{\sc R.~Altmann, P.~Henning, and D.~Peterseim}, {\em Numerical homogenization
  beyond scale separation}, Acta Numerica, 30 (2021), pp.~1--86.

\bibitem{ambartsumyan2020hierarchical}
{\sc I.~Ambartsumyan, W.~Boukaram, T.~Bui-Thanh, O.~Ghattas, D.~Keyes,
  G.~Stadler, G.~Turkiyyah, and S.~Zampini}, {\em Hierarchical matrix
  approximations of {H}essians arising in inverse problems governed by {PDE}s},
  SIAM Journal on Scientific Computing, 42 (2020), pp.~A3397--A3426.

\bibitem{Badia2020}
{\sc S.~Badia and F.~Verdugo}, {\em Gridap: An extensible finite element
  toolbox in {Julia}}, Journal of Open Source Software, 5 (2020), p.~2520,
  \url{https://doi.org/10.21105/joss.02520},
  \url{https://doi.org/10.21105/joss.02520}.

\bibitem{boulle2021learning}
{\sc N.~Boull{\'e} and A.~Townsend}, {\em Learning elliptic partial
  differential equations with randomized linear algebra}, Foundations of
  Computational Mathematics,  (2022), pp.~1--31.

\bibitem{chen2021multiscale}
{\sc J.~CHEN, F.~SCH{\"A}FER, J.~HUANG, and M.~DESBRUN}, {\em Multiscale
  {Cholesky} preconditioning for ill-conditioned problems},  (2021).

\bibitem{davis2016survey}
{\sc T.~A. Davis, S.~Rajamanickam, and W.~M. Sid-Lakhdar}, {\em A survey of
  direct methods for sparse linear systems}, Acta Numerica, 25 (2016),
  pp.~383--566.

\bibitem{de2021convergence}
{\sc M.~V. de~Hoop, N.~B. Kovachki, N.~H. Nelsen, and A.~M. Stuart}, {\em
  Convergence rates for learning linear operators from noisy data}, arXiv
  preprint arXiv:2108.12515,  (2021).

\bibitem{dekel2004bramble}
{\sc S.~Dekel and D.~Leviatan}, {\em The {Bramble--Hilbert} lemma for convex
  domains}, SIAM journal on mathematical analysis, 35 (2004), pp.~1203--1212.

\bibitem{fan2019bcr}
{\sc Y.~Fan, C.~O. Bohorquez, and L.~Ying}, {\em {BCR-Net}: A neural network
  based on the nonstandard wavelet form}, Journal of Computational Physics, 384
  (2019), pp.~1--15.

\bibitem{fan2019multiscale}
{\sc Y.~Fan, J.~Feliu-Faba, L.~Lin, L.~Ying, and L.~Zepeda-N{\'u}nez}, {\em A
  multiscale neural network based on hierarchical nested bases}, Research in
  the Mathematical Sciences, 6 (2019), pp.~1--28.

\bibitem{fan2019multiscaleH}
{\sc Y.~Fan, L.~Lin, L.~Ying, and L.~Zepeda-N{\'u}nez}, {\em A multiscale
  neural network based on hierarchical matrices}, Multiscale Modeling \&
  Simulation, 17 (2019), pp.~1189--1213.

\bibitem{feischl2018sparse}
{\sc M.~Feischl and D.~Peterseim}, {\em Sparse compression of expected solution
  operators}, arXiv preprint arXiv:1807.01741,  (2018).

\bibitem{husfeldt2015graph}
{\sc T.~Husfeldt}, {\em Graph colouring algorithms}, Encyclopedia of
  Mathematics and its Applications, Cambridge University Press, 2015,
  p.~277–303, \url{https://doi.org/10.1017/CBO9781139519793.016}.

\bibitem{levitt2022randomized}
{\sc J.~Levitt and P.-G. Martinsson}, {\em Randomized compression of
  rank-structured matrices accelerated with graph coloring}, arXiv preprint
  arXiv:2205.03406,  (2022).

\bibitem{li2020fourier}
{\sc Z.~Li, N.~Kovachki, K.~Azizzadenesheli, B.~Liu, K.~Bhattacharya,
  A.~Stuart, and A.~Anandkumar}, {\em Fourier neural operator for parametric
  partial differential equations}, arXiv preprint arXiv:2010.08895,  (2020).

\bibitem{li2020multipole}
{\sc Z.~Li, N.~Kovachki, K.~Azizzadenesheli, B.~Liu, K.~Bhattacharya,
  A.~Stuart, and A.~Anandkumar}, {\em Multipole graph neural operator for
  parametric partial differential equations}, arXiv preprint arXiv:2006.09535,
  (2020).

\bibitem{li2020neural}
{\sc Z.~Li, N.~Kovachki, K.~Azizzadenesheli, B.~Liu, K.~Bhattacharya,
  A.~Stuart, and A.~Anandkumar}, {\em Neural operator: Graph kernel network for
  partial differential equations}, arXiv preprint arXiv:2003.03485,  (2020).

\bibitem{lin2011fast}
{\sc L.~Lin, J.~Lu, and L.~Ying}, {\em Fast construction of hierarchical matrix
  representation from matrix--vector multiplication}, Journal of Computational
  Physics, 230 (2011), pp.~4071--4087.

\bibitem{lin2009fast}
{\sc L.~Lin, J.~Lu, L.~Ying, R.~Car, E.~Weinan, et~al.}, {\em Fast algorithm
  for extracting the diagonal of the inverse matrix with application to the
  electronic structure analysis of metallic systems}, Communications in
  Mathematical Sciences, 7 (2009), pp.~755--777.

\bibitem{lin2011selinv}
{\sc L.~Lin, C.~Yang, J.~C. Meza, J.~Lu, L.~Ying, and W.~E}, {\em
  {SelInv}---{An} algorithm for selected inversion of a sparse symmetric
  matrix}, ACM Transactions on Mathematical Software (TOMS), 37 (2011),
  pp.~1--19.

\bibitem{lu2019deeponet}
{\sc L.~Lu, P.~Jin, and G.~E. Karniadakis}, {\em {DeepONet}: Learning nonlinear
  operators for identifying differential equations based on the universal
  approximation theorem of operators}, arXiv preprint arXiv:1910.03193,
  (2019).

\bibitem{martinsson2008rapid}
{\sc P.-G. Martinsson}, {\em Rapid factorization of structured matrices via
  randomized sampling}, arXiv preprint arXiv:0806.2339,  (2008).

\bibitem{martinsson2016compressing}
{\sc P.-G. Martinsson}, {\em Compressing rank-structured matrices via
  randomized sampling}, SIAM Journal on Scientific Computing, 38 (2016),
  pp.~A1959--A1986.

\bibitem{nelsen2021random}
{\sc N.~H. Nelsen and A.~M. Stuart}, {\em The random feature model for
  input-output maps between banach spaces}, SIAM Journal on Scientific
  Computing, 43 (2021), pp.~A3212--A3243.

\bibitem{owhadi2020ideas}
{\sc H.~Owhadi}, {\em Do ideas have shape? plato's theory of forms as the
  continuous limit of artificial neural networks}, arXiv preprint
  arXiv:2008.03920,  (2020).

\bibitem{owhadi2019operator}
{\sc H.~Owhadi and C.~Scovel}, {\em Operator-Adapted Wavelets, Fast Solvers,
  and Numerical Homogenization: From a Game Theoretic Approach to Numerical
  Approximation and Algorithm Design}, vol.~35, Cambridge University Press,
  2019.

\bibitem{schafer2020sparse}
{\sc F.~Sch\"{a}fer, M.~Katzfuss, and H.~Owhadi}, {\em Sparse {C}holesky
  factorization by {K}ullback-{L}eibler minimization}, SIAM J. Sci. Comput., 43
  (2021), pp.~A2019--A2046, \url{https://doi.org/10.1137/20M1336254},
  \url{https://doi.org/10.1137/20M1336254}.

\bibitem{schafer2021compression}
{\sc F.~Sch\"{a}fer, T.~J. Sullivan, and H.~Owhadi}, {\em Compression,
  inversion, and approximate {PCA} of dense kernel matrices at near-linear
  computational complexity}, Multiscale Model. Simul., 19 (2021), pp.~688--730,
  \url{https://doi.org/10.1137/19M129526X},
  \url{https://doi.org/10.1137/19M129526X}.

\bibitem{stepaniants2021learning}
{\sc G.~Stepaniants}, {\em Learning partial differential equations in
  reproducing kernel {Hilbert} spaces}, arXiv preprint arXiv:2108.11580,
  (2021).

\end{thebibliography}

\appendix

\section{Proofs of theoretical results}

\subsection{Prior results}
Our proofs of \cref{thm:main_thm_simplicial,thm:main_thm_supernodal,thm:low_rank} rely on results from \cite{schafer2021compression}.
The first result bounds extremal eigenvalues of $\KM$ and its Schur complements.

\begin{theorem}[Spectral localization {\cite[Theorem 5.16]{schafer2021compression}}]
  \label{thm:spectral_loc}
  In the setting of \cref{thm:main_thm_simplicial,thm:main_thm_supernodal,thm:low_rank} there exists a constant $C$, depending only on $\|\IK\|$, $\|\IK^{-1}\|$, $d$, $s$, and $\delta$ such that writing $\bar{\I}^{(k)} \coloneqq \cup_{1 \leq l \leq k} \I^{(l)}$, we have for each $1 \leq k < l \leq q$ that 
  \begin{align*}
    \frac{1}{C} h^{2ks} &\leq \lambda_{\min}\left(\KM_{\bar{\I}^{(l)}, \bar{\I}^{(k)}}\right), \quad  \frac{1}{C} h^{2qs} \leq \lambda_{\min}\left(\KM\right)\\
    C h^{2ks} &\geq \lambda_{\max}\left(\KM_{\I^{(l)}, \I^{(l)}} - \KM_{\I^{(l)}, \bar{\I}^{(k)}} \left(\KM_{\bar{\I}^{(k)}, \bar{\I}^{(k)}}\right)^{-1} \KM_{\bar{\I}^{(k)}, \I^{(l)}} \right),
  \end{align*}
  Where $\lambda_{\min}$ and $\lambda_{\max}$ denote the smallest respectively largest value of their (symmetric and positive-definite) arguments.
\end{theorem}

The second result describes the exponential decay of the Cholesky factors of $\KM$ that enables the exponential accuracy of our approximation.

\begin{theorem}[Exponential decay {\cite[Theorem 5.23]{schafer2021compression}}]
  \label{thm:spatial_loc}
  In the setting of \cref{thm:main_thm_simplicial,thm:main_thm_supernodal,thm:low_rank} there exist constants $C$, $\alpha$, and $\gamma$, depending only on $\|\IK\|$, $\|\IK^{-1}\|$, $d$, $s$, $h$, and $\delta$, such that the entries of the exact lower triangular Cholesky factor $\dL$ of $\KM$ satisfy
  \begin{equation*}
    \left|\dL_{ij}\right| \leq C N^{\alpha}  \exp\left(- h^{-(\min(k,l))}\gamma \dist\left(t(\dw_{i}), t(\dw_{j})\right)\right)
  \end{equation*}
\end{theorem}

\subsection{Proof of \texorpdfstring{\cref{thm:main_thm_simplicial}}.}
\label{sec:proof_simplicial}
We denote as $\S_\rho \subset \I \times \I$ the sparsity pattern of $\dL$ as produced by \cref{alg:Cholesky}. 
For $i \succ j$, the entry $\dL_{ij}$ is part of the sparsity pattern if and only if $t(\dw_{i})$ is closest to $t(\dw_{j})$ out of all other basis functions of the same color.
For $i, \iota \in c \subset \colors^{(k)}, i \neq \iota$ sharing the same color, we have $\dist\left(t(\dw_{i}), t(\dw_{\iota})\right) \geq \rho 2 h^{k - 1}$. 
Together with the triangle inequality, this implies that for $i \in \I^{(k)}$ and $j \in \I^{(l)}$,
\begin{equation}
  \label{eqn:sparsity_decay}
j \preceq i, (i, j) \notin \S_{\rho} \Rightarrow \dist\left(t(\dw_{i}), t(\dw_{j})\right) \geq \rho h^{l-1} - h^{k-1} \geq (\rho - 1) h^{l - 1}.  
\end{equation}
For $S \subset \I \times \I$, the truncation operator $A \mapsto \trunc_{\S}(A)$ takes in a matrix $A \in \I \times \I$ and returns a matrix $\trunc_{\S}(A)$ that is equal to $A$ on $\S$ and equal to zero, everywhere else.
We can then write 
\begin{equation*}
  \KM = \underbrace{\trunc_{\S_{\rho}}(\dL) \left(\trunc_{\S_{\rho}}(\dL)\right)^{\top}}_{\bar{\KM} \coloneqq} + \underbrace{\KM - \bar{\KM}}_{-E \coloneqq}. 
\end{equation*} 
By completing the square and using \cref{thm:spatial_loc}, \cref{thm:spectral_loc}, and \cref{eqn:sparsity_decay}, we obtain
\begin{footnotesize}
\begin{align*}
  \left\|E\right\| &= \left\| \left(\dL - \trunc_{\S_{\rho}}\left(\dL\right)\right) \dL^{\top} + \dL \left(\dL - \trunc_{\S_{\rho}}\left(\dL\right)\right)^{\top} + \left(\dL - \trunc_{\S_{\rho}}\left(\dL\right)\right) \left(\dL - \trunc_{\S_{\rho}}\left(\dL\right)\right)^{\top}\right\|\\
  &\leq \bar{C} N^{\bar{\alpha}}  \exp\left(- \rho \bar{\gamma} \right),
\end{align*}
\end{footnotesize}
where $\bar{C}$, $\bar{\alpha}$, and $\bar{\gamma}$ have the same dependences as the original $C$, $\alpha$, and $\gamma$.
It is left to show that applying \cref{alg:Cholesky} $\bar{\KM}$ correctly recovers the \emph{exact} Cholesky $\trunc_{\S_{\rho}}\left(\dL\right)$ of $\bar{\KM}$.
This can be proved by induction over the for-loop of \cref{alg:Cholesky}.
Assume that for a given color $c \in \colors$, the columns corresponding to indices of all prior colors were reconstructed correctly. 
Denote as $k$ the index of the last column that has been reconstructed already and denote as $c_\succ$ the color following $c$. 
By assumption, the argument of $\texttt{scatter}$ can then be reformulated as
\begin{align*}
  \dO_{:, c} - \dL \dL^{\top} \dM_{:, c} 
  =& \bar{\KM} \dM_{:,c} - \left(\bar{\KM}_{:, 1:k} \left(\bar{\KM}_{1 : k, 1 : k}\right)^{-1} \bar{\KM}_{1:k, :}\right) \dM_{:, c} \\
  =& \underbrace{\left( \bar{\KM} - \bar{\KM}_{:, 1:k} \left(\bar{\KM}_{1 : k, 1 : k}\right)^{-1} \bar{\KM}_{1:k, :}\right)}_{\schur^{k} \coloneqq}\dM_{:, c} .
\end{align*}
It is well known that the $(k+1)$-st column of the \emph{Schur complement} $\schur^k$ is equal to the $k+1$-st column of the Cholesky factor, up to a diagonal scaling. 
Furthermore, we observe that eliminating the $(k+1)$-st row and column does not impact any rows and columns of $\schur^k$ that are part of the same color $c^{+}$ since, in general, an entry $(i,j)$ can not be in the sparsity set $\S_\rho$ if $i$ and $j$ are not in the same color.
In particular, the nonzeros on columns in $c^{+}$ are located on the sparsity set $\S_{\rho}$ and thus multiplication with $\dM_{:, c^{+}}$ followed by $\texttt{scatter}$ recovers these columns of $\schur^k$ exactly. 
After diagonal normalization, we thus exactly recover the columns of $\dL$ in color $c^{+}$.
The result then follows by induction over the colors in $\colors$.

\subsection{Proof of \texorpdfstring{\cref{thm:main_thm_supernodal}}.}
\label{sec:proof_supernodal}
We now use the limited number of supernodal colors to rigorously bound the error propagation in \cref{alg:SuperCholesky}, closely following the proof of \cite[Theorem 5.27]{schafer2021compression}.
In the following, we assume that all supernodes in a given color $\tilde{c}$ have the same size and thus sums of the form $\sum_{\tilde{i} \in \tilde{c}}\KM_{\tilde{i}\tilde{j}}$ are well-defined. 
Our proof can easily be extended beyond this case by truncating or zero-padding supernodal entries accordingly, at the cost of additional notational complexity.

Analogously to the proof of \cref{thm:main_thm_simplicial}, we denote as $\tilde{\S}_{\rho} \subset \tilde{\I} \times \tilde{\I}$ the supernodal sparsity pattern of the Cholesky factor computed by \cref{alg:SuperCholesky}, and let $\trunc_{\tilde{\S}_{\rho}}$ denote the operation of truncating a given supernodal block-matrix to this sparsity pattern. 
As before, we then denote $\bar{\KM} = \KM + E$ as a perturbation of the Cholesky factorization that has \emph{exactly} sparse Cholesky factors according to the pattern $\tilde{\S}$.

We will show that there exist constants $\bar{C}, \bar{\alpha} > 0$ depending only on $\|\IK\|$, $\|\IK^{-1}\|$, $d$, $s$, $h$, and $\delta$ such that for $\epsilon < \bar{C}^{-1} N^{-\bar{\alpha}}$ and $\|E\| \leq \epsilon$, \cref{alg:SuperCholesky} applied to $\KM$ terminates without encountering a non- positive-definite diagonal block and such that the resulting Cholesky factor $\dL$ satisfies $\|\dL \dL^{\top} - \bar{\KM}\| \epsilon C N^{\bar{\alpha}}$.

We prove this by controlling how successive columns of the Cholesky factor $\dL$ computed by \cref{alg:SuperCholesky} diverge from those of the factors $\bar{\dL}$ of the perturbed matrix $\bar{\KM}$ during the factorization. 
We introduce the following notation. 
For $\tilde{c} \in \supercolors$, we denote additive updates given by the $\tilde{c}$-th column of the Cholesky factors as
\begin{equation*}
  \bar{\addschur}^{\tilde{c}} \coloneqq \bar{\dL}_{:, \tilde{c}}\bar{\dL}_{:, \tilde{c}}^{\top} 
  \quad 
  \text{and}
  \quad 
  \addschur^{\tilde{c}} \coloneqq \dL_{:, \tilde{c}}\dL_{:, \tilde{c}}^{\top}.
\end{equation*}
We also introduce notation for the (approximate) Schur complements 
\begin{equation*}
  \schur^{\tilde{c}} \coloneqq \KM - \sum \limits_{\tilde{c}^- \preceq\tilde{c}}  \addschur^{\tilde{c}^-}, 
  \quad 
  \bar{\schur}^{\tilde{c}} \coloneqq \bar{\KM} - \sum \limits_{\tilde{c}^- \preceq\tilde{c}}  \bar{\addschur}^{\tilde{c}^-}. 
\end{equation*}
While $\bar{\schur}^{\tilde{c}}$ coincides with the exact Schur complement of $\bar{\KM}$ after eliminating the colors up to $\tilde{c}$ and thus coincides with $\sum_{\tilde{c}^+ \succ \tilde{c}} \bar{\addschur}^{\tilde{c}^+}$, $\schur^{\tilde{c}}$ only approximates the Schur complement of $\KM$, and is thus only approximately equal to $\sum_{\tilde{c}^+ \succ \tilde{c}} \addschur^{\tilde{c}^+}$.
The key result that allows us to prove \cref{thm:main_thm_supernodal} is the following lemma that bounds the rate at which $\schur^{\tilde{c}}$ and $\bar{\schur}^{\tilde{c}}$ diverge as we iterate through the supernodal colors $\tilde{c} \in \supercolors$.
It will be used recursively to provide upper bounds $\epsilon^{(k)}$ for the error on each scale $k$, in terms of the $\big\{\epsilon^{(l)}\big\}_{0 \leq l < k}$ on coarser scales.

\begin{lemma}
  \label{lem:error_propagation}
  There exists a constant $g$ depending only on $d$, such that for $\rho \geq 1$, the following holds.

  For a given color $\tilde{c} \in \supercolors^{(p)}$ let $\left\{\epsilon^{(k)}\right\}_{0 \leq k \leq q}$ be such that for each $1 \leq k \leq q$ and $\tilde{c}^{-} \in \supercolors^{(k)}$ with $\tilde{c}^{-} \preceq \tilde{c}$, we have
  
  \noindent \begin{minipage}{0.56\linewidth}
   \begin{equation}
    \label{eqn:propagation_assumption_1}
  \max \limits_{\tilde{c}^{-} \preceq \tilde{\xi}, \tilde{\chi} \in \supercolors} \max \limits_{\tilde{i} \in \tilde{\xi}, \tilde{j} \in \chi} \left\| \addschur^{\tilde{c}^{-}}_{\tilde{i}\tilde{j}} - \bar{\addschur}^{\tilde{c}^{-}}_{\tilde{i}\tilde{j}} \right\|_{\fro} \leq \epsilon^{(k)} \quad \text{and} \quad
    \end{equation}     
  \end{minipage}
  \begin{minipage}{0.40\linewidth}
    \begin{equation}
        \label{eqn:propagation_assumption_2}
        \max \limits_{\tilde{i}, \tilde{j} \in \tilde{\I}} \left\| \KM_{\tilde{i} \tilde{j}} - \bar{\KM}_{\tilde{i} \tilde{j}} \right\|_{\fro} \leq \epsilon^{(0)}.
    \end{equation}
  \end{minipage}\par\vspace{\belowdisplayskip}
  Denoting as $\|\cdot\|$ the operator norm and as $\tilde{c}^{+}$ the direct successor of $\tilde{c}$, we define
  \begin{equation*}
    \epsilon \coloneqq g^2 \left( \sum \limits_{k = 0}^{p} h^{d(p - k)} \epsilon^{(k)} \right), \quad
    \lambda_{\min} \coloneqq \lambda_{\min} \left(\bar{\KM}_{\tilde{c}^+ \tilde{c}^+}\right), \quad \text{and} \quad \lambda_{\max} \coloneqq \max_{\tilde{c}^+ \preceq \tilde{\xi} \in \supercolors} \left\|\bar{\schur}^{\tilde{c}}_{\tilde{c}^{+} \tilde{\xi}}\right\|.
  \end{equation*}
   Assume that $\epsilon \leq \frac{\lambda_{\min}}{2}$.
  We then have
  \begin{equation*}
    \max \limits_{\tilde{c} \preceq \tilde{\xi}, \tilde{\chi} \in \supercolors} \max \limits_{\tilde{i} \in \tilde{\xi}, \tilde{j} \in \chi} \left\| \schur^{\tilde{c}^{+}}_{\tilde{i} \tilde{j}} - \bar{\schur}^{\tilde{c}^{+}}_{\tilde{i} \tilde{j}} \right\|_{\fro} \leq \left(1 + \frac{3}{2} \frac{\lambda_{\max}}{\lambda_{\min}} + \frac{9}{2} \frac{\lambda_{\max}^2}{\lambda_{\min}^2}\right) \epsilon.
  \end{equation*}
\end{lemma}

To prove \cref{lem:error_propagation} we will use the following geometric lemma
\begin{lemma}
  \label{lem:ball_packing} 
  There exists a constant $g$ depending only on $d$, such that 
  \begin{enumerate}
    \item The number of colors on each scale of the hierarchy is bounded by $g$
    \item For $1 \leq k \leq l \leq q$ and $\tilde{c}^{(k)} \in \supercolors^{(k)}$ and $\tilde{c}^{(l)} \in \supercolors^{(l)}$, $g$ bounds the sizes of the
  \begin{equation*}
    a^{\tilde{c}^{(k)}, \tilde{c}^{(l)}}_{\tilde{i}} \coloneqq \left\{\tilde{j} \in \tilde{c}^{(l)} \text{ such that }  \addschur^{\tilde{c}^{(k)}}_{\tilde{i}\tilde{j}} \text{ or } \bar{\addschur}^{\tilde{c}^{(k)}}_{\tilde{i}\tilde{j}}\neq 0 \right\}
  \end{equation*}
  as $\# a^{\tilde{c}^{(k)}, \tilde{c}^{(l)}}\left(\tilde{i}\right) \leq g h^{(k - l)}$, as well as $\# \tilde{c}^{(l)} \leq g h^{-(l - 1)}$.
  \end{enumerate} 
\end{lemma}
\begin{proof}[Proof of \cref{lem:ball_packing}]
  To obtain the first result that this number can be bounded in terms of only $d$, we first note that by construction, the aggregation centers $\left\{y_{\tilde{i}}\right\}_{\tilde{i} \in \tilde{\I}^{(k)}}$ of supernodes on a given scale $k$ have a pairwise distance of at least $2 \rho h^{k - 1}$. 
  At the same time, for each $i \in \tilde{i}$  we have $\dist\left(t(\dw_i), y_{\tilde{i}}\right) \leq 2 \rho h^{k - 1}$, since otherwise there would have to exist a $\tilde{j} \in \tilde{\I}^{(k)}$ with $\dist\left(t(\dw_{i}, y_{\tilde{j}})\right) \leq \dist\left(t(\dw_{i}, y_{\tilde{i}})\right)$ and thus $i$ would be aggregated into $\tilde{j}$ instead of $\tilde{i}$. 
  Since furthermore each $\left\{t(\dw_{i})\right\}_{i \in \tilde{\I}^(k)}$ is contained in a ball of radius $h^{k - 1}$, each $\left\{t(\dw_i)\right\}_{i \in \tilde{i}}$ is contained in a ball of radius $3 \rho h^{k - 1}$ centered in $y_{\tilde{i}}$.
  In particular, for any two supernodes $\tilde{i},\tilde{j} \in \tilde{\I}^{(k)}$ that are not allowed to have the same color have to satisfy $\dist\left(y_{\tilde{i}}, y_{\tilde{j}}\right) \leq (3 + 3 + 1) \rho h^{k - 1}$.
  Let now $G = (V, E)$ be the undirected graph with vertices $V = \tilde{\I}^{(k)}$ and edges given by pairs $\left\{\tilde{i}, \tilde{j}\right\}$ of supernodes that satisfy $\dist\left(y_{\tilde{i}}, y_{\tilde{j}}\right) \leq (3 + 3 + 1) \rho h^{k - 1}$. 
  By the above, any graph coloring of this graph yields an admissible supernodal coloring. 
  Furthermore, the number of colors needed to color a graph $G$ with degree $\Delta(G)$ using a greedy algorithm is upper bounded \cite{husfeldt2015graph} by $\Delta(G) + 1$. 
  For every $\left\{\tilde{i}, \tilde{j}\right\} \in E$, a ball of radius $\rho h^{k - 1}$ centered in $y_{\tilde{j}}$ must be contained in a ball of radius $(3 + 3 + 1 + 1) \rho h^{k - 1}$.
  Since the $\left\{y_{\tilde{i}} \right\}_{\tilde{i} \in \tilde{\I}^{(k)}}$ have a pairwise distance of at least $\rho h^{k - 1}$, these balls are disjoint.
  Therefore, the maximal number of edges incident on a given node is given by the ratio of the volumes of a ball of radius $8 \rho h^{k - 1}$ and one of radius $\rho h^{k - 1}$, and thus the degree is bounded as $\Delta(G) \leq 8^d$.
  The number of colors is therefore bounded by $8^d + 1$.

  The second result can be shown (for a possibly larger $g$) using a very similar ball-packing argument as above, where we note that the range of interaction of $\schur^{(\tilde{c}^{(k)})}$ is upper bounded as $\lessapprox \rho h^{k - 1}$, while the distance between elements of $\tilde{c}^{(l)}$ is lower bounded by $\gtrapprox \rho h^{l - 1}$ and thus the estimation of the above quantity amounts to counting the disjoint balls of radius $\gtrapprox \rho h^{l - 1}$ that can be fit into a ball of radius $\lessapprox \rho h^{k - 1}$.
\end{proof}
Equipped with \cref{lem:ball_packing}, we can now proceed to prove  \cref{lem:error_propagation}.
\begin{proof}[Proof of \cref{lem:error_propagation}]
  The $\tilde{c}^+$-th column of the Cholesky factors can in turn be expressed as 
  \begin{align*}
    &\bar{\dL}_{\tilde{i}\tilde{j}} = \bar{\schur}^{\tilde{c}}_{\tilde{i}\tilde{j}} \left(\bar{\schur}^{\tilde{c}}_{\tilde{j}\tilde{j}}\right)^{-1/2}, 
    \quad
    &\dL_{\tilde{i}\tilde{j}} = \left(\sum \limits_{\tilde{\iota}\in \tilde{c}^{+}} \schur^{\tilde{c}}_{\tilde{i}\tilde{\iota}}\right) \operatorname{sym}\left(\sum \limits_{\tilde{\iota}\in \tilde{c}^{+}} \schur^{\tilde{c}}_{\tilde{j}\tilde{\iota}}\right)^{-1/2} \\ 
    &\bar{\dL}_{\tilde{j}\tilde{j}} = \left(\bar{\schur}^{\tilde{c}}_{\tilde{j}\tilde{j}}\right)^{1/2}, 
    \quad
    &\dL_{\tilde{i}\tilde{j}} = \operatorname{sym}\left(\sum \limits_{\tilde{\iota}\in \tilde{c}^{+}} \schur^{\tilde{c}}_{\tilde{j}\tilde{\iota}}\right)^{1/2} 
  \end{align*}
  with $\operatorname{sym}(A) = (A + A^{\top}) / 2$ denoting the symmetrization operation. For $(\tilde{i},\tilde{j})\notin \tilde{\S}$, both factors are equal to zero and for $\tilde{i} = \tilde{j}$, they are given by

  We can therefore write 
  \begin{alignat*}{3}
    \bar{\addschur}^{\tilde{c}^+}_{\tilde{i}\tilde{j}} 
      &= \bar{\schur}^{\tilde{c}}_{\tilde{i}\tilde{\eta}} \left(\bar{\schur}^{\tilde{c}}_{\tilde{\eta}\tilde{\eta}}\right)^{\text{-}1}\bar{\schur}^{\tilde{c}}_{\tilde{\eta}\tilde{j}}, 
    \quad &\addschur^{\tilde{c}^+}_{\tilde{i}\tilde{j}} 
      &= \left(\sum \limits_{\tilde{\iota}\in \tilde{c}^{+}} \schur^{\tilde{c}}_{\tilde{i}\tilde{\iota}}\right) \operatorname{sym}\left(\sum \limits_{\tilde{\iota}\in \tilde{c}^{+}} \schur^{\tilde{c}}_{\tilde{\eta}\tilde{\iota}}\right)^{\text{-}1} \left(\sum \limits_{\tilde{\iota}\in \tilde{c}^{+}} \schur^{\tilde{c}}_{\tilde{\iota}\tilde{j}}\right)\\ 
    \bar{\addschur}^{\tilde{c}^+}_{\tilde{j}\tilde{j}} 
      &= \left(\bar{\schur}^{\tilde{c}}_{\tilde{j}\tilde{j}}\right), 
    \quad& \addschur^{\tilde{c}^+}_{\tilde{j}\tilde{j}} 
      &= \operatorname{sym}\left(\sum \limits_{\tilde{\iota}\in \tilde{c}^{+}} \schur^{\tilde{c}}_{\tilde{j}\tilde{\iota}}\right),
  \end{alignat*}
  whenever $\texttt{scatter}$ assigns $\tilde{i}$ and $\tilde{j}$ to the same column, $\tilde{\eta} = \tilde{\eta}(\tilde{i}) = \tilde{\eta}(\tilde{j})$. 
  Otherwise, the off-diagonal terms above are zero. 

  Writing again $\tilde{\eta}(\tilde{i})$ for the index of $\tilde{c}$ that $\tilde{i}$ is scattered into, we observe that 
  \begin{equation*}
    \bar{\schur}^{\tilde{c}}_{\tilde{i}\tilde{\eta}(\tilde{i})} = \sum \limits_{\tilde{\iota}\in \tilde{c}^{+}} \schur^{\tilde{c}}_{\tilde{i}\tilde{\iota}}, 
    \quad 
    \bar{\schur}^{\tilde{c}}_{\tilde{\eta}(\tilde{j})\tilde{j}} = \sum \limits_{\tilde{\iota}\in \tilde{c}^{+}} \schur^{\tilde{c}}_{\tilde{\iota}\tilde{j}}, 
  \end{equation*} 
  since the Schur complement $\bar{\schur}^{\tilde{c}}$ is equal to $\sum_{\tilde{c}^+ \succ \tilde{c}}\bar{\dL}_{:, \tilde{c}}\left(\bar{\dL}_{:, \tilde{c}}\right)^{\top}$.
  Defining 
  \begin{equation*}
    X_{\tilde{i} \tilde{j}} \coloneqq 
    \begin{cases}
       \sum \limits_{\tilde{\iota}\in \tilde{c}^{+}} \schur^{\tilde{c}}_{\tilde{\iota}\tilde{j}} - \bar{\schur}^{\tilde{c}}_{\tilde{\iota}\tilde{j}} \quad &\text{ for } \quad \tilde{i} \neq \tilde{j} \quad \text{ and } \quad \tilde{i} = \tilde{\eta}\left(\tilde{j}\right), \\
       \sum \limits_{\tilde{\iota}\in \tilde{c}^{+}} \schur^{\tilde{c}}_{\tilde{i}\tilde{\iota}} - \bar{\schur}^{\tilde{c}}_{\tilde{i}\tilde{\iota}} \quad &\text{ for } \quad \tilde{i} \neq \tilde{j} \quad \text{ and } \quad \tilde{j} = \tilde{\eta}\left(\tilde{i}\right), \\
      \operatorname{sym}\left(\sum \limits_{\tilde{\iota}\in \tilde{c}^{+}} \schur^{\tilde{c}}_{\tilde{\iota}\tilde{j}} - \bar{\schur}^{\tilde{c}}_{\tilde{\iota}\tilde{j}}\right) \quad &\text{ for } \quad \tilde{i} = \tilde{j} \in \tilde{c}^{+},\\
      0 \quad &\text{ else},
    \end{cases}
  \end{equation*}
  and using the classical matrix identity $\left(A + B\right) = A^{-1} - \left(A + B\right)^{-1} B A^{-1}$, we obtain
  \begin{align*}
    \schur^{\tilde{c}^{+}}_{\tilde{i}\tilde{j}} 
      =& \left(\sum \limits_{\tilde{\iota}\in \tilde{c}^{+}} \schur^{\tilde{c}}_{\tilde{i}\tilde{\iota}}\right) \operatorname{sym}\left(\sum \limits_{\tilde{\iota}\in \tilde{c}^{+}} \schur^{\tilde{c}}_{\tilde{\eta}\tilde{\iota}}\right)^{-1} \left(\sum \limits_{\tilde{\iota}\in \tilde{c}^{+}} \schur^{\tilde{c}}_{\tilde{\iota}\tilde{i}}\right)\\
      =& 
      \left(\bar{\schur}_{\tilde{i}\tilde{\eta}}^{\tilde{c}} + X_{\tilde{i} \tilde{\eta}}\right) 
      \left(\bar{\schur}^{\tilde{c}}_{\tilde{\eta}\tilde{\eta}} + X_{\tilde{\eta}\tilde{\eta}}\right)^{-1} 
      \left(\bar{\schur}^{\tilde{c}}_{\tilde{\eta}\tilde{j}} + X_{\tilde{\eta}\tilde{j}}\right) \\
      =& \bar{\schur}_{\tilde{i}\tilde{\eta}}^{\tilde{c}} \left(\bar{\schur}^{\tilde{c}}_{\tilde{\eta}\tilde{\eta}}\right)^{-1} \bar{\schur}^{\tilde{c}}_{\tilde{\eta}\tilde{j}}
          + \bar{\schur}_{\tilde{i}\tilde{\eta}}^{\tilde{c}} \left(\bar{\schur}^{\tilde{c}}_{\tilde{\eta}\tilde{\eta}}\right)^{-1} X_{\tilde{\eta}\tilde{j}}
          + X_{\tilde{i} \tilde{\eta}} \left(\bar{\schur}^{\tilde{c}}_{\tilde{\eta}\tilde{\eta}}\right)^{-1} \bar{\schur}^{\tilde{c}}_{\tilde{\eta}\tilde{j}}
          + X_{\tilde{i} \tilde{\eta}} \left(\bar{\schur}^{\tilde{c}}_{\tilde{\eta}\tilde{\eta}}\right)^{-1} X_{\tilde{\eta}\tilde{j}}\\
        & + \left(\bar{\schur}_{\tilde{i}\tilde{\eta}}^{\tilde{c}} + X_{\tilde{i} \tilde{\eta}}\right) 
      \left(\bar{\schur}^{\tilde{c}}_{\tilde{\eta}\tilde{\eta}} + X_{\tilde{\eta}\tilde{\eta}}\right)^{-1} X_{\tilde{\eta}\tilde{\eta}} \left(\bar{\schur}^{\tilde{c}}_{\tilde{\eta}\tilde{\eta}}\right)^{-1}
      \left(\bar{\schur}^{\tilde{c}}_{\tilde{\eta}\tilde{j}} + X_{\tilde{\eta}\tilde{j}}\right). 
  \end{align*}
  In other words, we have 
  \begin{equation}
    \label{eqn:schur_change_off_diag}
    \begin{split}
        \schur^{\tilde{c}^{+}}_{\tilde{i}\tilde{j}} - \bar{\schur}^{\tilde{c}^{+}}_{\tilde{i}\tilde{j}}
        =& 
        \bar{\schur}_{\tilde{i}\tilde{\eta}}^{\tilde{c}} \left(\bar{\schur}^{\tilde{c}}_{\tilde{\eta}\tilde{\eta}}\right)^{-1} X_{\tilde{\eta}\tilde{j}}
                + X_{\tilde{i} \tilde{\eta}} \left(\bar{\schur}^{\tilde{c}}_{\tilde{\eta}\tilde{\eta}}\right)^{-1} \bar{\schur}^{\tilde{c}}_{\tilde{\eta}\tilde{j}}
                + X_{\tilde{i} \tilde{\eta}} \left(\bar{\schur}^{\tilde{c}}_{\tilde{\eta}\tilde{\eta}}\right)^{-1} X_{\tilde{\eta}\tilde{j}}\\
              & + \left(\bar{\schur}_{\tilde{i}\tilde{\eta}}^{\tilde{c}} + X_{\tilde{i} \tilde{\eta}}\right) 
            \left(\bar{\schur}^{\tilde{c}}_{\tilde{\eta}\tilde{\eta}} + X_{\tilde{\eta}\tilde{\eta}}\right)^{-1} X_{\tilde{\eta}\tilde{\eta}} \left(\bar{\schur}^{\tilde{c}}_{\tilde{\eta}\tilde{\eta}}\right)^{-1}
            \left(\bar{\schur}^{\tilde{c}}_{\tilde{\eta}\tilde{j}} + X_{\tilde{\eta}\tilde{j}}\right). 
    \end{split}
  \end{equation}

  Similarly, for the diagonal entries, we have 
  \begin{equation}
      \label{eqn:schur_change_diag}
      \schur^{\tilde{c}^{+}}_{\tilde{j}\tilde{j}}
        = \operatorname{sym}\left(\sum \limits_{\tilde{\iota}\in \tilde{c}^{+}} \schur^{\tilde{c}}_{\tilde{\eta}\tilde{\iota}}\right) = \bar{\schur}^{\tilde{c}}_{\tilde{j}\tilde{j}} + X_{\tilde{j}\tilde{j}}
  \end{equation}
  and thus 
  \begin{equation*}
    \schur^{\tilde{c}^{+}}_{\tilde{j}\tilde{j}} - \bar{\schur}^{\tilde{c}^{+}}_{\tilde{j}\tilde{j}}
    = X_{\tilde{j}\tilde{j}}.
  \end{equation*}
  In order to conclude the proof of the lemma, we need to show and exploit that $X$ is small. 
  For $\tilde{c} \in \supercolors^{(p)}$ we can write
  \begin{align*}
    \sum \limits_{\tilde{\iota}\in \tilde{c}^{+}} \schur^{\tilde{c}}_{\tilde{\iota}\tilde{j}} - \bar{\schur}^{\tilde{c}}_{\tilde{\iota}\tilde{j}} 
    &= \left(\sum \limits_{\tilde{\iota}\in \tilde{c}^{+}} \bar{\KM}_{\tilde{\iota}\tilde{j}} -  \KM_{\tilde{\iota}\tilde{j}}\right)
    + \sum \limits_{k = 1}^{p} \sum \limits_{\subalign{\tilde{c}^- &\preceq \tilde{c} \\ \tilde{c}^- &\in \supercolors^{(k)}}} \left(\sum \limits_{\tilde{\iota}\in \tilde{c}^{+}} \addschur^{\tilde{c}}_{\tilde{\iota}\tilde{j}} - \bar{\addschur}^{\tilde{c}}_{\tilde{\iota}\tilde{j}} \right)\\
    &= \left(\sum \limits_{\tilde{\iota}\in \tilde{c}^{+}} \bar{\KM}_{\tilde{\iota}\tilde{j}} -  \KM_{\tilde{\iota}\tilde{j}}\right)
    + \sum \limits_{k = 1}^{p} \sum \limits_{\subalign{\tilde{c}^- &\preceq \tilde{c} \\ \tilde{c}^- &\in \supercolors^{(k)}}} \left(\sum \limits_{\tilde{\iota}\in a^{\tilde{c}^{-}, \tilde{c}^{+}}_{\tilde{i}}} \addschur^{\tilde{c}}_{\tilde{\iota}\tilde{j}} - \bar{\addschur}^{\tilde{c}}_{\tilde{\iota}\tilde{j}} \right).
  \end{align*}
  Using \cref{eqn:propagation_assumption_1,eqn:propagation_assumption_2} as well as \cref{lem:ball_packing} and the triangle inequality, we obtain
  \begin{equation*}
    \left\|\sum \limits_{\tilde{\iota}\in \tilde{c}^{+}} \schur^{\tilde{c}}_{\tilde{\iota}\tilde{j}}- \bar{\schur}^{\tilde{c}}_{\tilde{\iota}\tilde{j}} \right\|_{\fro}
    \leq 
    g h^{-p} \epsilon^{0} + g^2 \sum \limits_{k = 1}^{p} h^{k - p} \epsilon^{(k)} \leq g^2 \sum \limits_{k = 0}^{p} h^{k - p} \epsilon^{(k)}.
  \end{equation*}
  Since $\operatorname{sym}(\cdot)$ is an orthogonal projection with respect to the Frobenius inner product, this implies that for all $\tilde{i}, \tilde{j}$ we have $X_{\tilde{i}\tilde{j}} \leq \epsilon$.
  Plugging this estimate into \cref{eqn:schur_change_off_diag,eqn:schur_change_diag}, the triangle inequality and submultiplicativity $\|AB\|_{\fro} \leq \|A\| \|B\|_{\fro}$ imply
  \begin{equation*}
    \left\|\schur^{\tilde{c}^{+}}_{\tilde{i}\tilde{j}} - \bar{\schur}^{\tilde{c}^{+}}_{\tilde{i}\tilde{j}}\right\|_{\fro}
    \leq 2 \frac{\lambda_{\max}}{\lambda_{\min}} \epsilon + \frac{\epsilon}{\lambda_{\min}} \epsilon + \frac{9}{2} \frac{\lambda_{\max}^2}{\lambda_{\min}^2} \epsilon 
    \leq \left(1 +\frac{3}{2} \frac{\lambda_{\max}}{\lambda_{\min}} + \frac{9}{2} \frac{\lambda_{\max}^2}{\lambda_{\min}^2} \right)\epsilon. 
  \end{equation*}
\end{proof}

We now use \cref{lem:error_propagation} construct a sequence of upper bounds  $\epsilon^{(p)}$ on successive scales and use it to upper bound a weighted sum $\epsilon^{(\preceq p)}$ of all errors up to this scale.
\begin{lemma}
  \label{lem:single_scale_propagation}
  Assume that for $0 \leq k \leq p$ the $\epsilon^{(k)}$ satisfy \cref{eqn:propagation_assumption_1} and \cref{eqn:propagation_assumption_2} for any $\tilde{c} \in \cup_{1 \leq k \leq p} \supercolors^{(k)}$.
  Let $\lambda_{\min}$ and $\lambda_{\max}$ be such that for each $\tilde{c} \in \supercolors^{(p)} \cup \supercolors^{(p + 1)}$ and $\tilde{c} \preceq \tilde{c}^{+} \in \supercolors^{(p + 1)}$ they satisfy 
  \begin{equation*}
    \lambda_{\min} \leq \lambda_{\min} \left(\bar{\KM}_{\tilde{c}^{+} \tilde{c}^{+}}\right), \quad \text{and} \quad \lambda_{\max} \geq \max_{\tilde{c}^+ \preceq \tilde{\xi} \in \supercolors} \left\|\bar{\schur}^{\tilde{c}}_{\tilde{c}^+ \tilde{\xi}}\right\|.
  \end{equation*}
  Define for $g$ as in \cref{lem:ball_packing},
  \begin{equation*}
    \epsilon^{\left(\preceq p\right)} \coloneqq g^2 \left( \sum \limits_{k = 0}^{p} h^{d(p - k)} \epsilon^{(k)} \right), 
    \quad
    \varphi \coloneqq \left(1 + \frac{3}{2} \frac{\lambda_{\max}}{\lambda_{\min}} + \frac{9}{2} \frac{\lambda_{\max}^2}{\lambda_{\min}^2} \right).
  \end{equation*}
  Then, if
  \begin{equation*}
  \epsilon^{(p + 1)} \coloneqq \frac{1 - \left(\varphi g^2\right)^{g}}{1 - \varphi g^2} \varphi g^2 \epsilon^{(\preceq p)} \leq \frac{\lambda_{\min}}{2}, 
  \end{equation*}
  it also satisfies 
  \begin{equation*}
    \max \limits_{\tilde{c} \preceq \tilde{\xi}, \tilde{\chi} \in \supercolors} \max \limits_{\tilde{i} \in \tilde{\xi}, \tilde{j} \in \chi} \left\| \schur^{\tilde{c}^{+}}_{\tilde{i} \tilde{j}} - \bar{\schur}^{\tilde{c}^{+}}_{\tilde{i} \tilde{j}} \right\|_{\fro} \leq \left(1 + \frac{3}{2} \frac{\lambda_{\max}}{\lambda_{\min}} + \frac{9}{2} \frac{\lambda_{\max}^2}{\lambda_{\min}^2}\right) \epsilon^{(p + 1)}
  \end{equation*}
  for any $\tilde{c}^{+} \in \supercolors^{(p + 1)}$.
\end{lemma}
\begin{proof}
  For $\tilde{c} \in \supercolors^{(p + 1)}$, let $\epsilon^{\left(p + 1\right), \preceq \tilde{c}}$ satisfy, 
  for each $\tilde{c}^{-} \prec \tilde{c}, \tilde{c}^- \in \supercolors^{(p+1)}$, 
  \begin{equation*}
    \max \limits_{\tilde{c}^{-} \preceq \tilde{\xi}, \tilde{\chi} \in \supercolors} \max \limits_{\tilde{i} \in \tilde{\xi}, \tilde{j} \in \chi} \left\| \addschur^{\tilde{c}^{-}}_{\tilde{i}\tilde{j}} - \bar{\addschur}^{\tilde{c}^{-}}_{\tilde{i}\tilde{j}} \right\|_{\fro} \leq \epsilon^{(p + 1), \preceq \tilde{c}}.
  \end{equation*}
  \Cref{lem:error_propagation} then implies that for $\tilde{c}^+$ the direct successor of $\tilde{c}$ and
  \begin{equation*}
    \epsilon^{(p + 1), \preceq \tilde{c}^+} \coloneqq \varphi g^2 \left(\epsilon^{(\preceq p)} + \epsilon^{(p + 1), \preceq \tilde{c}}\right), 
  \end{equation*}
  $\epsilon^{(p + 1), \preceq \tilde{c}^+}$ satisfies, for each $\tilde{c}^- \preceq \tilde{c}^+, \tilde{c}^- \in \supercolors^{(p + 1)}$,
  \begin{equation*}
    \max \limits_{\tilde{c}^{-} \preceq \tilde{\xi}, \tilde{\chi} \in \supercolors} \max \limits_{\tilde{i} \in \tilde{\xi}, \tilde{j} \in \chi} \left\| \addschur^{\tilde{c}^{-}}_{\tilde{i}\tilde{j}} - \bar{\addschur}^{\tilde{c}^{-}}_{\tilde{i}\tilde{j}} \right\|_{\fro} \leq \epsilon^{(p + 1), \preceq \tilde{c}}.
  \end{equation*}
  Thus, proving the lemma amounts to upper bounding the $\lceil g \rceil$-th entry of the recursion
  $x_{k + 1} = \phi g^2 \left(\epsilon^{(\preceq p)} + x_{k}\right)$, with initial entry $x_0 = 0$.
  By induction, one can show that $x_{k} = \sum_{l = 1}^k \left(\phi g^2\right)^l \epsilon^{(\preceq p)} = \frac{1 - \left(\phi g^2\right)^{k}}{1 - \phi g^2} \phi g^2 \epsilon^{(\preceq p)}$. The result follows by setting $k = g$.
\end{proof}

With \cref{lem:error_propagation,lem:single_scale_propagation} in place, we can now proceed to prove \cref{thm:main_thm_supernodal}.
\begin{proof}[Proof of \cref{thm:main_thm_supernodal}]
  We use $\gtrapprox, \lessapprox$ to denote inequality up to constants that are subsumed in the $\O$ notation in the theorem.
  By \cref{thm:spectral_loc}, for $\rho \gtrapprox \log(N)$  there exists a constant $\kappa$ depending only on $\|\IK\|$, $\|\IK^{-1}\|$, $d$, $s$, $\delta$, and $h$ such that for any $\tilde{c} \in \supercolors$ and $\lambda_{\min}$, $\lambda_{\max}$ as defined in \cref{lem:error_propagation,lem:single_scale_propagation} we have $\lambda_{\max} / \lambda_{\min} \leq \kappa$. 
  Defining, analogue to \cref{lem:single_scale_propagation}, $\varphi \coloneqq \left(1 + 3 \kappa / 2 + 9/2 \kappa^2\right)$.
  Then, \cref{lem:single_scale_propagation} implies that for any $1 \leq p \leq q$ for which $\{\epsilon^{(k)}\}_{0 \leq k \leq p}$ satisfy \cref{eqn:propagation_assumption_1,eqn:propagation_assumption_2} and $\epsilon^{(\preceq p)}$ as defined in \cref{lem:single_scale_propagation}, 
  $\epsilon^{(p + 1)} \coloneqq \frac{1 - \left(\varphi g^2\right)^{g}}{1 - \varphi g^2} \varphi g^2 \epsilon^{(\preceq p)} \leq \frac{\lambda_{\min}}{2}$ satisfies \cref{eqn:propagation_assumption_1,eqn:propagation_assumption_2}, as well.
  The resulting $\epsilon^{(\preceq p + 1)}$ is then related to $\epsilon^{\preceq p}$ by 
  \begin{equation*}
    \epsilon^{(\preceq p + 1)} = h^{-1} \epsilon^{(\preceq p)} + \varphi g^4 \frac{1 - (\varphi g^2)^g}{1 - \varphi g^2} \epsilon^{(\preceq p)} = \left(h^{-1} + \varphi g^4 \frac{1 - (\varphi g^2)^g}{1 - \varphi g^2}\right) \epsilon^{(\preceq p)}.
  \end{equation*}
  By choosing $\rho \gtrapprox \log(N)$ we can make $\epsilon^{(0)}$ small enough to ensure that 
  \begin{equation*}
    \epsilon \coloneqq \epsilon^{(\preceq q)} = \left(h^{-1} + \varphi g^4 \frac{1 - (\varphi g^2)^g}{1 - \varphi g^2}\right) \epsilon^{(0)} \leq \frac{\lambda_{\min}\left(\bar{\KM}\right)}{2}
  \end{equation*}
  and thus $\left\|\bar{\dL} \bar{\dL}^{\top} - \dL \dL^{\top}\right\| \leq \poly(N) \epsilon^{(0)}$.
  The accuracy result then follows from \cref{thm:main_thm_simplicial} that proves the exponential decay of $\epsilon^{(0)}$ in $\rho$.
  The computational complexity result follows each row of $\dL$ having only $\O\left(\log(N) \rho^{d}\right)$ nonzero entries.
\end{proof}

\subsection{Proof of \texorpdfstring{\cref{thm:low_rank}}.}
\label{sec:proof_low_rank}
We conclude this section with a proof of \cref{thm:low_rank}.
The low-rank approximation rate was already established by \cite{schafer2021compression}. We extend this result by explicitly proving the optimality of the low-rank approximation.  
\begin{proof}[Proof of \cref{thm:low_rank}]
  \Cref{thm:spectral_loc} directly implies that for $1 \leq p \leq q$, 
  \begin{equation*}
    \left\| \KM - \dL_{:, \bar{\I}^{p}} \left(\dL_{:, \bar{\I}^{p}}\right)^{\top}\right\| \leq C \|\KM\|  h^{2 p s}
  \end{equation*}
  for a constant $C$ depending only on $\|\IK\|$, $\|\IK^{-1}\|$, $d$, $h$, $\delta$, and $s$.
  By possibly changing this constant, we also have $C^{-1} h^{-pd} \leq \# \bar{\I}^{(p)} \leq C h^{-pd}$, which implies the approximation rate claimed in \cref{thm:low_rank}.
  We note that \cref{thm:spectral_loc} implies that $\KM$ has, for each $1 \leq p \leq q$, a submatrix of size at least $C^{-1} h^{-pd}$ with minimal eigenvalue lower bounded by $C^{-1} h^{pd}$. 
  Thus, a rank $C^{-1} h^{-pd}$ approximation can be at best of accuracy $C^{-1} h^{pd}$, establishing optimality of the approximation rate given in \cref{thm:low_rank}.
\end{proof}

\subsection{Proof of \texorpdfstring{\cref{thm:continuous_operator}}.}
We begin by analyzing the case $\rho = \infty$, in the more general case where the $v$ and $w$ are obtained, not from piecewise constant functions, but from piecewise polynomials. 

\begin{lemma}
  \label{lem:approx}
  In the setting of \cref{thm:main_thm_simplicial,thm:main_thm_supernodal}, modify the construction of the $V^{(k)}$ and $W^{(k)}$ as follows. 
  Instead of the $v \in V^{(k)}$ being the piecewise constant functions on $\tau^{(k)}$, define them to be the piecewise polynomials of order $p - 1$. 
  As before, the $w_{i}$ are chosen to be a local orthonormal basis of the orthogonal complement of  $V^{(k - 1)}$ in $V^{(k)}$.
  For an $s \leq r \leq 2s$, we assume furthermore that $\IK$ is a bounded and invertible map, from $H^{r} \cap H^s_{0}\left(\Omega\right)$ to $H^{r - 2s}\left(\Omega\right)$ 
  We then have, for a constant $C$ depending only on $d, \Omega, \IK, p$, for any $\tilde{g}$ in the $L^2$-orthogonal complement of $V^{(k)}$,
  \begin{equation*}
    \left\|\K \tilde{g}\right\|_{L^2} \leq C h^{\min(p, r)k} \left\|\tilde{g}\right\|_{L^2}.
  \end{equation*}
\end{lemma}

\begin{proof}
  For any $\tilde{g}$ in the orthogonal complement of $V^{q}$, we have
  \begin{align*}
    \left\|\K \tilde{g}\right\|_{L^2} 
    =& \sup_{v \in L^2\left(\Omega\right)} \frac{\left\langle v, \K \tilde{g}\right\rangle_{L^2}}{\|v\|_{L^2}}
    \leq \sup_{v \in H^{r - 2s}\left(\Omega\right)} \frac{\left\langle v, \K \tilde{g}\right\rangle_{L^2}}{\|v\|_{H^{r - 2s}}}
    = \sup_{v \in H^{r - 2s}\left(\Omega\right)} \frac{\left\langle \K v, \tilde{g}\right\rangle_{L^2}}{\|\IK \K v\|_{H^{r - 2s}}}\\
    =& \sup_{v \in H^{r} \cap H^s_0} \frac{\left\langle v, \tilde{g}\right\rangle_{L^2}}{\|\IK v\|_{H^{r - 2s}}}
    \lessapprox \sup_{v \in H^{r}\left(\Omega\right)} \frac{\left\langle v, \tilde{g}\right\rangle_{L^2}}{\|v\|_{H^{r}}}
    = \sup_{v \in H^r\left(\Omega\right)} \frac{\sum_{t \in \tau^{(k)}}\left\langle v - p_t, \tilde{g}\right\rangle_{L^2(t)}}{\|v\|_{H^r}}\\
    \leq&  \sup_{v \in H^r\left(\Omega\right)} \frac{\sum_{t \in \tau^{(k)}}\left\langle v - p_t, \tilde{g}\right\rangle_{L^2(t)}}{\|v\|_{H^r}}
    = \sup_{v \in H^r\left(\Omega\right)} \frac{\left\langle \sum_{t \in \tau^{(k)}} (v\middle|_{t} - p_t), \tilde{g}\right\rangle_{L^2}}{\|v\|_{H^r}}\\
    \leq& \sup_{v \in H^r\left(\Omega\right)} \frac{\left\| \sum_{t \in \tau^{(k)}} (v\middle|_{t} - p_t)\right\|_{L^2} \left\| \tilde{g}\right\|_{L^2(t)}}{\|v\|_{H^r}}.
  \end{align*}
  Here, $v|_t$ denotes the restriction of $v$ to $t$ and each $p_t$ is an arbitrary polynomial of order $(\min(p,r) - 1)$, continued with zero outside $t$.
  For appropriately chosen $p_t$, the  Bramble-Hilbert lemma \cite{dekel2004bramble} implies that 
  \begin{align*}
    & \left\| \sum_{t \in \tau^{(k)}} (v|_{t} - p_t)\right\|_{L^2}^2 
    = \sum \limits_{t \in \tau^{(k)}} \big\| v|_{t} - p_t\big\|_{L^2(t)}^2 \\
     \leq& C h^{2k\min(p, r)} \sum \limits_{t \in \tau^{(k)}} \big\| v|_{t}\big\|_{H^{\min(p,r)}(t)}^2 
    \leq C h^{2k\min(p, r)} \big\| v\big\|_{H^r}^2.
  \end{align*}
  Thus, we have $\left\|\K \tilde{g}\right\|_{L^2} \leq C h^{\min(p, r)k} \left\|\tilde{g}\right\|_{L^2}$.
  With the estimates above, we have 
  \begin{equation*}
    \left \langle f, \left(\K - \bar{\K}_{\infty}\right) g \right \rangle_{L^2}
    \leq Ch^{\min(p, r)k} \|f\|_{L^2} \|g\|_{L^2} = Ch^{pk} \|f\|_{L^2} \|g\|_{L^2} = 
  \end{equation*} 
\end{proof}

\begin{corollary}
  \label{cor:opnorm}
  In the setting of \cref{lem:approx}, define the projected Green's function $\bar{\K}$ as  
  \begin{equation*} 
    \bar{G}(x, y) = \sum_{i \in \cup_{k \leq q} \I^{(k)}} \sum_{j \in \cup_{k \leq q} \I^{(k)}} \left\langle w_{i}, \K w_{j} \right\rangle_{L^2} w_{i}(x) w_{j}(y).
  \end{equation*} 
  We then have, for a constant $C$ depending only on $d, \Omega, \IK, \delta, p$,
  \begin{equation*}
   \|\bar{\K} - \K \|_{L^2 \rightarrow L^2} \leq C h^{\min(p,r)q}. 
  \end{equation*}
  Here $\| \cdot \|_{L^2 \rightarrow L^2}$ denotes the operator norm with respect $L^{2}\left(\Omega\right)$ and we interpret the true and approximate Green's function as operators on $L^2\left(\Omega\right)$, by convolution. 
\end{corollary}
\begin{proof}
  Writing $f = \bar{f} + \tilde{f}$ and $g = \bar{g} + \tilde{g}$ for $\bar{f}, \bar{g}\in L^2\left(\Omega\right)$ piecewise constant on elements of $\tau^{(q)}$, we have
  \begin{align*}
    &\left \langle f, \left(\K - \bar{\K}_{\infty}\right) g \right \rangle_{L^2}
    = \left \langle \bar{f} + \tilde{f}, \left(\K - \bar{\K}_{\infty}\right) \left(\bar{g} + \tilde{g}\right) \right \rangle_{L^2} \\
    =& \left \langle \bar{f}, \left(\K - \bar{\K}_{\infty}\right) \tilde{g} \right \rangle_{L^2}  
    + \left \langle \tilde{f}, \left(\K - \bar{\K}_{\infty}\right) \bar{g} \right \rangle_{L^2} 
    + \left \langle \tilde{f}, \left(\K - \bar{\K}_{\infty}\right) \tilde{g} \right \rangle_{L^2}  \\
    =& \left \langle \bar{f}, \K \tilde{g} \right \rangle_{L^2}
    + \left \langle \tilde{f}, \K \bar{g} \right \rangle_{L^2}
    + \left \langle \tilde{f}, \K \tilde{g} \right \rangle_{L^2}
    \leq 2 \left(\big\|f\big\|_{L^2} \big\|\K \tilde{g} \big\|_{L^2} + \big\|g\big\|_{L^2} \big\|\K \tilde{f} \big\|_{L^2}\right)\\
    \leq& C h^{\min(p,r)q} \big\|f\big\|_{L^2}  \big\|g\big\|_{L^2}.
  \end{align*} 
\end{proof} 

We can now conclude the proof of \cref{thm:continuous_operator}.

\begin{proof}[Proof of \cref{thm:continuous_operator}]
  We use the triangle inequality together with \cref{cor:opnorm} to write 
  \begin{equation*}
    \left\| \K - \bar{\K}_{\rho} \right\|_{L^2 \rightarrow L^2} 
    \leq \left\| \K - \bar{\K}\right\|_{L^2 \rightarrow L^2} + \left\| \bar{\K} - \bar{\K}_{\rho} \right\|_{L^2 \rightarrow L^2} 
    \leq Ch^q + \left\| \bar{\K} - \bar{\K}_{\rho} \right\|_{L^2 \rightarrow L^2}.
  \end{equation*}
\end{proof}
By \cref{thm:main_thm_supernodal}, $\left\| \bar{\K} - \bar{\K}_{\rho} \right\|_{L^2 \rightarrow L^2}$ can be upper bounded as
\begin{equation*}
  \left\| \sum_{i \in \cup_{k \leq q} \I^{(k)}} \sum_{j \in \cup_{k \leq q} \I^{(k)}} \left(\KM - \dL \dL^{\top}  \right)_{ij} w_{i}(x) w_{j}(y) \right\|_{L^2 \rightarrow L^2} = \left\| \KM - \dL \dL^{\top}\right\| \leq h^{q},
\end{equation*}
by choosing $\rho \gtrapprox \log(N h^{-q}) \approx \log(N)$.
By observing that the number $N$ of basis functions scales as $N \approx h^{-qd}$, we obtain the desired result. 
\subsection{Proof of \cref{thm:continuous_l2}}
We begin by studying the rate of convergence in Hilbert-Schmidt-norm obtained from piecewise polynomial approximation. 

\begin{corollary}
  \label{cor:hsnorm}
  Consider the setting of \cref{lem:approx} with infinitely fine nested partitions $\left(\tau^{(k)}\right)_{1 \leq k < \infty}$ of $\Omega$. 
  For $1 \leq q$, define the projected Green's function $\bar{\K}$ as  
  \begin{equation*} 
    \bar{G}(x, y) = \sum_{i \in \cup_{k \leq q} \I^{(k)}} \sum_{j \in \cup_{k \leq q} \I^{(k)}} \left\langle w_{i}, \K w_{j} \right\rangle_{L^2} w_{i}(x) w_{j}(y)
  \end{equation*} 
  We then have, for a constant $C$ depending only on $d, \Omega, \IK, h, \delta, p$,
  \begin{equation*}
   \|\bar{\K} - \K \|_{L^2 \otimes L^2} \leq C h^{(2\min(p, r) - d)q} \approx C \left(\# \cup_{1 \leq k \leq q} \I^{(k)}\right)^{-\frac{(2\min(p, r) / d - 1) }{ 2}}. 
  \end{equation*}
  In particular, an $\epsilon$-accurate $L^2(\Omega \times \Omega)$ approximation of the Green's function can be obtained using $C\epsilon^{\frac{2}{2 \min(p, r) / d - 1}}$ matrix-vector products.
\end{corollary}
\begin{proof}
  Note that when interpreting $\bar{\K} - \K$ as operators mapping $L^2$ to itself, the norm $\| \cdot \|_{L^2(\Omega \times \Omega)}$ is equal to their Hilbert-Schmidt norm.
  It can thus be computed as 
  \begin{equation}
    \|\bar{\K} - \K \|_{L^2(\Omega \times \Omega)}^2 = \sum_{i} \left\|\left(\bar{\K} - \K\right) w_i\right\|^2_{L^2} \leq C \sum_{k > q} \# \I^{(k)} h^{2\min(p, r)k}.   
  \end{equation}
  By a ball packing argument, we have $\# \I^{(k)} \approx \leq C h^{-kd}$, resulting in 
  \begin{align*}
    &\|\bar{\K} - \K \|_{L^2(\Omega \times \Omega)}^2 \leq C \sum_{k > q} h^{(2\min(p, r) - d)k}  \\
    =& \frac{C}{ 1 - h^{(2\min(p, r) - d)}} h^{(2\min(p, r) - d)q}\leq C h^{(2\min(p, r) - d)q}.
  \end{align*}
  From $q \approx -\log_h( \# \I^{(q)}) / d$, we obtain $\|\bar{\K} - \K \|_{L^2 \otimes L^2}^2 \leq C \left(\#\tau^{(q)}\right)^{-(2 \min(p, r) / d - 1)}.$
\end{proof} 
For $p = 1$ and $d \geq 2$, we have $2 \min(p, r) / d - 1 \leq 0$. Thus, \cref{cor:hsnorm} does not apply to the case of basis functions obtained from averages over elements of $\tau$.
Green's function approximations with bounds in $L^2$ need an additional approximation step.
\begin{lemma}
  \label{lem:interpol}
  In the setting of \cref{cor:hsnorm} denote as $\bar{\K}$, the Green's function obtained from setting $p = 1$ for a given $q$. For a $1 < \hat{q} < q$ and an arbitrary $p$, denote as $\left\{\hat{w}_{i}\right\}_{i \in \hat{I}}$ a local multiscale orthogonal basis of the space of functions that are polynomials of order $p - 1$ on each element of $\tau^{\hat{q}}$, denoting as $\hat{W}$ their span. 
  Denote as $\hat{\I}^{(k)}$ the indices corresponding the basis function on the $k$-th scale and denote as $\hat{W}^{(k)}$ the span of these basis functions. 
  Let $\hat{\bar{\K}}$ be the $L^2$ projection of $\bar{\K}$ onto $\hat{W}$, whereby convolution with $\hat{\bar{\K}}$ amounts to convolution with $\bar{\K}$, pre-- and postprocessed with $L^2$ projection onto $\hat{W}$.
  We then have, for a constant $C$ depending only on $d, \Omega, \IK, h, \delta, p$
  \begin{equation*}
   \|\hat{\K} - \K \|_{L^2 \otimes L^2} \leq C h^{(2\min(p, r) - d)q} \approx C \left(\# \tau^{(q)}\right)^{-\left(1 - \frac{d}{2\min(p, r)}\right)}.
  \end{equation*}
\end{lemma}

\begin{proof}
  Define the Green's function $\hat{\K}$ projected onto the $\hat{W}$ as 
  \begin{equation*}
    \hat{G}(x, y) = \sum_{\hat{i} \in \hat{\I}} \sum_{\hat{j} \in \hat{\I}} \left\langle \hat{w}_{\hat{i}}, \K \hat{w}_{\hat{j}} \right\rangle_{L^2} \hat{w}_{\hat{i}}(x) \hat{w}_{\hat{j}}(y).
  \end{equation*}
  Using the triangle inequality and \cref{cor:hsnorm}, we obtain
  \begin{equation*}
    \|\hat{\bar{\K}} - \K \|_{L^2\otimes L^2} 
    \leq \|\hat{\bar{\K}} - \hat{\K}\|_{L^2\otimes L^2} + \|\hat{\K} - \K  \|_{L^2\otimes L^2}
    \leq \|\hat{\bar{\K}} - \hat{\K}\|_{L^2\otimes L^2} + C h^{\left(\min(p, r) - \frac{d}{2}\right)\hat{q}}
  \end{equation*}
  In order to bound $\|\hat{\bar{\K}} - \hat{\K}\|_{L^2\otimes L^2}$, we first observe that for any $\hat{w}_i, \hat{w}_j$ we have 
  \begin{equation*} 
    \left\| \left( \hat{\bar{\K}} - \hat{\K}\right) \hat{w}_{i}\right\|_{L^2}
    \leq 
    \begin{cases}
      0 \quad &\mathrm{if} \quad \text{$i \in \hat{I}^{(k)}$ for a $k > \hat{q}$,} \\ 
      \left\| \left( \bar{\K} - \K\right) \hat{w}_{i}\right\|_{L^2} \quad &\mathrm{else}.
    \end{cases}
  \end{equation*}
  The first case follows, since the projection of $\hat{w}_i$ onto $\hat{W}^{(\hat{q})}$ is zero. 
  In the second case, this projection onto $\hat{W}^{(\hat{q})}$ is the identity. 
  Meanwhile, the orthogonal projection of $\left(\bar{\K} - \K\right) \hat{w}_i$ onto $\hat{W}^{(\hat{q})}$ only decreases the $L^2$-norm of the result, proving the estimate for the second case.
  \cref{cor:opnorm} implies $\big\| \big( \hat{\bar{\K}} - \hat{\K}\big) \hat{w}_{i}\big\|_{L^2} \leq C h^{\min(p, r)q}$ and thus
  \begin{equation*}
    \left\| \left( \hat{\bar{\K}} - \hat{\K}\right) \right\|_{L^2\otimes L^2} 
    \leq C \sqrt{\sum \limits_{1 \leq k \leq \hat{q}} \# \left(\hat{\I}^{(k)}\right) h^{2q}}
    \leq C h^{q - \hat{q} d / 2} 
  \end{equation*}
  setting $\lceil q / \min(p, r) \rceil \geq \hat{q} \geq \lfloor q / \min(p, r) \rfloor $, we thus obtain
  \begin{equation*}
    \|\hat{\bar{\K}} - \K \|_{L^2 \otimes L^2} 
    \leq C h^{(\min(p, r) - d / 2)\hat{q}} 
    = C h^{\left(1 - \frac{d}{2\min(p, r)}\right){q}}
    \leq 
    C \left(\# \tau^{(q)}\right)^{-\left(1 - \frac{d}{2\min(p, r)}\right)}.
  \end{equation*}
  Here, the last line follows by the same argument as in the proof of \cref{cor:hsnorm}.
\end{proof}

\begin{proof}[Proof of \cref{thm:continuous_l2}]
  Using the triangle inequality and \cref{lem:interpol}, we compute 
  \begin{align*}
    \left\| \hat{\bar{\K}}_{\rho, \hat{q}} - \K \right\|_{L^2 \otimes L^2} 
    &\leq \left\| \hat{\bar{\K}}_{\rho, \hat{q}} - \hat{\bar{\K}}_{\hat{q}} \right\|_{L^2 \otimes L^2}
        +\left\| \hat{\bar{\K}}_{\hat{q}} - \K \right\|_{L^2 \otimes L^2} \\
    &\leq \left\| \hat{\bar{\K}}_{\rho, \hat{q}} - \hat{\bar{\K}}_{\hat{q}} \right\|_{L^2 \otimes L^2}
        + C N^{-\left(1 - d / (2\min(p, r))\right)}.
  \end{align*}
\end{proof}
We now compute  
\begin{equation*}
  \left\| \hat{\bar{\K}}_{\rho, \hat{q}} - \hat{\bar{\K}}_{\hat{q}} \right\|_{L^2 \otimes L^2} \leq \sqrt{\sum \limits_{i \in \hat{I}} \left\| \left(\hat{\bar{\K}}_{\rho, \hat{q}} - \hat{\bar{\K}}_{\hat{q}}\right)\hat{w}_i \right\|_{L^2}}
  \leq N^{\frac{1}{2\min(p, r)}} \|\KM - \dL \dL^{\top}\|,  
\end{equation*} 
Where we have used the fact that $\# \hat{\I} \leq C N^{1/\min(p, r)}$. This term can be upper bounded by $C N^{-(1 - d / (2\min(p, r)))}$ by choosing $\rho \gtrapprox \log(N)$, proving the desired result.
\end{document}